\documentclass[12pt, reqno,amscd, psamsfonts]{amsart}
\usepackage{fullpage,url,amssymb,bm}
\usepackage{enumitem}
\usepackage[all]{xy} 
\usepackage{graphicx}
\usepackage{xcolor}
\usepackage{accents}

\usepackage[OT2,OT1]{fontenc}
\usepackage[latin1]{inputenc}
\usepackage{amsmath}
\usepackage{amsthm}
\usepackage{amssymb}
\usepackage{hyperref}

\font\elevensc=cmcsc10 scaled\magstephalf

\font\teneu=eufm10 scaled\magstep1\font\seveneu=eufm7\font\fiveeu=eufm5

\newfam\eufam  \def\eu{\fam\eufam\teneu}
\textfont\eufam=\teneu \scriptfont\eufam=\seveneu
                              \scriptscriptfont\eufam=\fiveeu

\newcommand{\ha}[1]{{\hbox to#1pt{}}}
\newcommand{\hb}[1]{{\hbox to-#1pt{}}}

\newcommand{\itm}[2]{\begin{itemize}[leftmargin=#1pt]{#2}\end{itemize}}

\newcommand{\nix}{{\phantom|}}

\newcommand{\hla}{\hookleftarrow}
\newcommand{\hra}{\hookrightarrow}

\newcommand{\rsdp}{{\,\times\kern-3pt\lower-1pt%
\hbox{$\scriptscriptstyle|$\ha3}}}

\newcommand{\srjr}{\twoheadrightarrow}

\newcommand{\dwnr}[2]{{\ha{#1}\downarrow\hb2\lower-5pt\hbox{$\scriptstyle #2}}}
\newcommand{\dwnl}[2]{{\lower-5pt\hbox{${\scriptstyle #2}$}\hb2\downarrow\ha{#1}}}

\newcommand{\hhb}[1]{\hbox to#1pt{}}
\newcommand{\hor}[1]{\smash
       {\mathop{{\lgrghtar}}\limits^{\lower2pt\hbox{$\scriptstyle #1$}}}}
\newcommand{\horr}[1]{\ha2\smash
      {\mathop{{\lglgrghtar}}\limits^{\lower2pt\hbox{$\scriptstyle #1$}}}\ha2}
\newcommand{\mpst}[1]{\,\smash
       {\mathop{{\mapsto}}\limits^{\lower2pt\hbox{$\scriptstyle #1$}}}\,}
\newcommand{\mpstt}[1]{\,\smash
       {\mathop{{\longmapsto}}\limits^{\lower2pt\hbox{$\scriptstyle #1$}}}\,}

\newcommand{\lgrghtar}{{\ha2{\relbar\joinrel\rightarrow}\ha2}}
\newcommand{\lglgrghtar}{{\ha1{\relbar\joinrel\relbar\joinrel\rightarrow}\ha1}}

\newcommand{\nmnm}[1]{{\scalebox{.85}[1.05]{\elevensc #1}}}
\newcommand{\NMNM}[1]{{\scalebox{1}[1]{\elevensc #1}}}

\newcommand{\plim}[1]{\hbox to14pt{\rm
lim\kern-14pt\lower4.5pt\hbox{$\scriptstyle\longleftarrow$}%
\kern-8pt\lower8.5pt\hbox{$\scriptstyle{#1}$}}\ha{3}}

\newcommand{\ilim}[1]{\hboxto14pt{\rm
lim\kern-14pt\lower4.5pt\hbox{$\scriptstyle\longrightarrow$}%
\kern-8pt\lower8.5pt\hbox{$\scriptstyle{#1}$}}\ha{3}}

\newcommand{\defi}[1]{\textsf{#1}}

\newcommand{\clC}{{\mathcal C}}
\newcommand{\clD}{{\mathcal D}}

\newcommand{\clI}{{\mathcal I}}
\newcommand{\clJ}{{\mathcal J}}
\newcommand{\clK}{{\mathcal K}}

\newcommand{\clO}{{\mathcal O}}
\newcommand{\clP}{{\mathcal P}}
\newcommand{\clQ}{{\mathcal Q}}

\newcommand{\clT}{{\mathcal T}}

\newcommand{\clV}{{\mathcal V}}
\newcommand{\clW}{{\mathcal W}}
\newcommand{\clX}{{\mathcal X}}
\newcommand{\clY}{{\mathcal Y}}
\newcommand{\clZ}{{\mathcal Z}}

\newcommand{\eua}{{\eu a}}

\newcommand{\eum}{{\eu m}}

\newcommand{\eut}{{\eu t}}

\newcommand{\lvA}{{\mathbb A}}

\newcommand{\lvC}{{\mathbb C}}

\newcommand{\lvF}{{\mathbb F}}

\newcommand{\lvP}{{\mathbb P}}
\newcommand{\lvQ}{{\mathbb Q}}
\newcommand{\lvR}{{\mathbb R}}

\newcommand{\lvZ}{{\mathbb Z}}

\DeclareMathOperator{\chr}{char}
\DeclareMathOperator{\codim}{codim}

\DeclareMathOperator{\Max}{Max}
\DeclareMathOperator{\Quot}{Quot}
\DeclareMathOperator{\res}{res}

\DeclareMathOperator{\Spec}{Spec}

\DeclareMathOperator{\td}{td}

\newcommand{\abs}[1]{#1^{^{\rm abs}}}

\newcommand{\Br}{{\rm Br}}

\newcommand{\bma}{\bm a\hb{7.5}\bm a}
\newcommand{\bmam}{\bm a\hb{5.25}\bm a}
\newcommand{\bmb}{\bm b\hb{6.5}\bm b}
\newcommand{\bmbm}{\bm b\hb{4}\bm b}

\newcommand{\bmdl}{{\bm\delta\hb{6.5}\bm\delta}}
\newcommand{\bmdlm}{{\bm\delta\hb{4.5}\bm\delta}}
\newcommand{\bmfm}{\bm f\hb{5.25}\bm f}
\newcommand{\bmSgodudlm}{\bm\Sigma\hb{10}\bm\Sigma_{\bmdlm}}
\newcommand{\bmt}{\bm t\hb{5}\bm t}
\newcommand{\bmtI}{\bm t\hb{5.5}\bm t_I}
\newcommand{\bmim}{{\bm\imath\hb{3.5}\bm\imath}}

\newcommand{\bmtm}{\bm t\hb{3.25}\bm t}

\newcommand{\bmu}{\bm u\hb{8}\bm u}
\newcommand{\bmum}{\bm u\hb{5.5}\bm u}

\newcommand{\bmx}{\bm x\hb{8}\bm x}

\newcommand{\cpm}{\hbox{$\ha1\scriptstyle\cup\ha2$}}
\newcommand{\cppr}[2]{#1\cpm\hb1\dts\cpm#2}

\newcommand{\DD}{{D}}
\newcommand{\DDf}{D_{\!f}}

\newcommand{\dts}{\hb1.\ha1.\ha1.\ha1}

\newcommand{\fbmt}{f\!,\bmt}
\newcommand{\fK}{{\!f_{\!K}}}
\newcommand{\ft}{{f\!,\ha{.5}\bmtm}}

\newcommand{\HHx}[3]{{\rm H}^{#1}\big(#2,\lvZ/n(#3)\big)}

\newcommand{\HH}{{\rm H}}

\newcommand{\II}[1]{I_{#1}}

\newcommand{\istp}[1]{\bm\vartheta_{\hb1#1}}

\newcommand{\JJ}[1]{J_{#1}}
\newcommand{\KC}{{\hbox{\rm\defi{(KC)}}}}
\newcommand{\Kv}{K\hb{1}v}
\newcommand{\Kvx}{K_{\veps}}

\newcommand{\kIx}[1]{k_{\scriptscriptstyle I_{#1}}}

\newcommand{\kov}{k_{0v}}
\newcommand{\ko}{k_{\scriptscriptstyle0}}
\newcommand{\koveps}{k_{\scriptscriptstyle0_{\veps}}}

\newcommand{\lit}[1]{\hbox to32pt{[#1]\hss}}

\newcommand{\lo}{l_{\scriptscriptstyle0}}

\newcommand{\lps}[1]{(\hb{1.5}(#1)\hb{1.5})}

\newcommand{\lvAt}{\lvA_{\bmtm}}
\newcommand{\lvAti}{\lvA_{\bmtm_i}}
\newcommand{\lvAi}{\lvA_{\ti}}
\newcommand{\lvAtiu}{\lvA_{\bmtm_{i+1}}}

\newcommand{\lvPIx}[1]{\lvP_{\!\scriptscriptstyle I_{#1}}}
\newcommand{\lvPf}{\lvP_{\!f}}
\newcommand{\lvPi}{\lvP_\ti}
\newcommand{\lvPft}{\lvP_{\!f\hb1,\ha1\bmtm}}

\newcommand{\lvPfU}[1]{\lvP_{\!f\!,\ha1U_{#1}}}

\newcommand{\lvPU}[1]{\lvP_{U_{#1}}}

\newcommand{\lvPt}{\lvP_{\bmtm}}

\newcommand{\lvPtp}{\lvP_{\bmtm'}}
\newcommand{\mic}[1]{{\scriptscriptstyle{\rm #1}}}

\newcommand{\nxx}{{\phantom.}}

\newcommand{\nuo}{{\nu-1}}
\newcommand{\oli}{\overline}

\newcommand{\Ro}{R_{\scriptscriptstyle0}}

\newcommand{\So}{S_{\scriptscriptstyle0}}
\newcommand{\Sod}{\Sigma_{\ha1\delta}}

\newcommand{\Th}{{\mathfrak{Th}}}
\newcommand{\Tft}{{\Theta_\ft}}
\newcommand{\Tfto}{{\thbar{\ha{1.5}\Theta}\hb2_\ft}}
\newcommand\thbar[1]{\accentset{\rule{.6em}
{1pt}}{#1^{\phantom,}\hb2}\ha1}\newcommand{\ti}{{t_i}}

\newcommand{\tl}[1]{\tilde{#1}}
\newcommand{\tlC}{{\tilde C}}  
\newcommand{\tlK}{{\tilde K}}  
\newcommand{\tlL}{{\tilde L}}  
\newcommand{\tlS}{{\tilde S}}  
\newcommand{\tlSo}{\tilde S_{\scriptscriptstyle0}}  

\newcommand{\tlUo}{\tilde U_{\scriptscriptstyle0}}  

\newcommand{\tlV}{{\tilde V}}  
\newcommand{\tlclX}{{\tilde{\clX\ha2}\!}}  
  
\newcommand{\tlclY}{{\tilde\clY\ha2}\!}  
  
\newcommand{\tlX}{{\tilde X}}  
\newcommand{\tlY}{{\tilde Y}}  
\newcommand{\tlZ}{{\tilde Z}}  
\newcommand{\tlk}{{\tilde k}}  

\newcommand{\tlko}{\tilde k_{\scriptscriptstyle0}}

\newcommand{\tlw}{{\tilde w}}  
\newcommand{\tlx}{{\tilde x}}

\newcommand{\tms}{^{\hb1\times}\!}

\newcommand{\tx}{t}
\newcommand{\ttt}{{\bm t}}

\newcommand{\Uaf}{{U_{\scriptscriptstyle0},\bmam,f}}
\newcommand{\Uabf}{{U_{\scriptscriptstyle0},\bmam,\bmbm,f}}

\newcommand{\UIx}[1]{U_{\hb1\scriptscriptstyle I_{#1}}}

\newcommand{\Unu}{U_{\hb1\nu}}
\newcommand{\Uo}{U_{\scriptscriptstyle0}}

\newcommand{\Ut}{U_{\bmtm}}

\newcommand{\Uu}[1]{U_{\!#1}}

\newcommand{\uo}{u_{\scriptscriptstyle0}}
\newcommand{\uop}{u'_{\scriptscriptstyle0}}

\newcommand{\Val}[1]{V_{#1}}

\newcommand{\Vu}[1]{V_{#1}}
\newcommand{\veps}{\varepsilon}
\newcommand{\vk}{v}

\newcommand{\vo}{{v}}

\newcommand{\voeps}{{v_{\scriptscriptstyle0_{\veps}}}}
\newcommand{\val}{\scalebox{.85}[1]{{\textbf{\textsf{val}}}}}

\newcommand{\vid}{{\lower-1pt\hbox{/}\kern-7pt{\hbox{\rm O}}}}

\newcommand{\wh}[1]{{\widehat #1}}

\newcommand{\xe}{x}
\newcommand{\xx}{x}

\theoremstyle{plain}

\newtheorem{theorem}{Theorem}[section]
\newtheorem*{maintheo*}{Main Theorem}
\newtheorem{corollary}[theorem]{Corollary}
\newtheorem{keylemma}[theorem]{Key Lemma}
\newtheorem{keylemmarev}[theorem]{Key Lemma (revisited)}
\newtheorem{lemma}[theorem]{Lemma}
\newtheorem{proposition}[theorem]{Proposition}

\theoremstyle{definition}

\newtheorem{example/fact}[theorem]{Example/Fact}

\newtheorem{hypot}[theorem]{Hypothesis}

\newtheorem{fact/definition}[theorem]{Fact/Definition}

\newtheorem{notations/remarks}[theorem]{Notations/Remarks}
\newtheorem{definitions/notations}[theorem]{Definitions/Notations}
\newtheorem{notations/facts}[theorem]{Notations/Facts}

\newtheorem{recipe}[theorem]{Recipe}
\newtheorem{remark/definition}[theorem]{Remark/Definition}
\newtheorem{definition/remarks}[theorem]{Definition/Remarks}

\usepackage[alphabetic,backrefs,lite]{amsrefs} 

\begin{document}

\title[Characterizing f.g.\ fields by a sentence] 
{\hfill $\scriptstyle\mathfrak{Peter\ha3Roquette\ha3zu\ha3
       seinem\ha3 90.\ha2Geburtstag\ha3gewidmet}$
\vskip15pt
{\scalebox{.95}[1]{Distinguishing every finitely 
generated field of} \\
characteristic $\neq2$ by a single field axiom$\ha2^{\bm\dag}$
\vskip5pt
{\scalebox{.85}[1]{\tiny\rm The\ha4strong\ha4Elementary%
\ha4equivalence\ha4vs\ha4Isomorphy\ha4Problem\ha4in\ha4%
Characteristic\ha4$\neq2$}}}}
%
%
%
%
%
%
\author{Florian Pop}

\address{Department of Mathematics, University of Pennsylvania
        \vskip0pt
        DRL, 209 S 33rd Street, Phila\-delphia, PA 19104, USA}
\email{pop@math.upenn.edu}
\urladdr{http://math.penn.edu/\~{}pop}

%
%
%
%
\begin{abstract}
We show that the isomorphy type of every finitely 
generated field $K$ with $\chr(K)\!\neq2$ is encoded 
by a \textit{\textbf{single\ha3explicit\ha3axiom}} 
$\istp K\!$ \textit{\textbf{in\ha3the\ha3language\ha3of\ha3fields}}, 
i.e., for all finitely generated 
fields $L$ one has: $\istp K$ holds in $L$ if and only 
if $K\cong L$ as fields. This extends earlier results 
by \scalebox{.8}[1]{\sc Julia Robinson, Rumely, 
Poonen, Scanlon}, the author, and others. 
\end{abstract}


\keywords{Elementary equivalence vs 
isomorphism, first-order definability, e.g. of valuations,
finitely generated fields, Milnor K-groups, Galois/\'etale 
cohomology, Kato's higher local-global Principles
\vskip2pt
$\hb2^\dag\,${\bf Note}: This manuscript is a revised version 
of~\scalebox{.8}[1]{\sc Pop}~[P5]. The main 
result of this manuscript was also announced by
\scalebox{.8}[1]{\sc Dittmann}~[Di] (using 
the same technical tools, but expending rather 
on \scalebox{.8}[1]{\sc Poonen}~[Po],
\scalebox{.8}[1]{\sc Pop}~[P4]).
}

%
%
\maketitle

%
\vskip-2\baselineskip
\section{Introduction}
%
%
%
%
%
%
We begin by recalling that a \defi{sentence}, or an
\defi{axiom} in the language of fields is any formula 
in the language of fields which has not free variables.
One denotes by $\Th(K)$ the set of all the sentences
in the language of fields which hold in a given field 
$K$. For instance, by mere definitions, the field 
axioms are part of $\Th(K)$ for every field $K$; 
further the fact that $K$ is algebraically closed, as 
well as $\chr(K)$ are encoded in $\Th(K)$. Namely,
$K$ is algebraically closed \ iff \ $K$ satisfies the
\defi{scheme of axioms of algebraically closed fields}
(asserting that every non-zero polynomial $p(T)$ 
over $K$ has a root in $K$); respectively one has
$\chr(K)=p\geqslant0$ \ iff \ $K$ satisfies the 
\defi{char$\,\,=p$ scheme of axioms} (asserting: 
$\chr=p>0$ \ iff \ $\Sigma_{\mic{i=1}}^{^p}1=0$, 
respectively $\chr=0$ \ iff \ $\Sigma_{\mic{i=1}}^{^n} 
1\neq0$ for all $n$). On the other hand, if $K:=\lvQ(t)$ 
is the rational function field in the variable $t$ 
over $\lvQ$, then the usual way to say that $t$ 
is transcendental over $\lvQ$, namely 
``\ha1{\it $p(t)\neq0$ for all non-zero polynomials 
$p(T)$ over $\lvQ\,$}'' is not a scheme of axioms 
in the language of fields (because $t$ is not part 
of the language of fields). 
\vskip2pt
Two fundamental general type results in algebra are 
the following:
\vskip2pt
\itm{15}{
\item[-] Algebraically closed fields $K,L$ have 
$\Th(K)=\Th(L)$ \ iff \ $\chr(K)=\chr(L)$.
\vskip2pt
\item[-] Arbitrary fields $K,L$ have $\Th(K)=\Th(L)$
\ iff \ there are isomorphic ultra-powers 
$^{^*}\hb3K\cong\ha0^{^*}\hb3L$.
}

\vskip2pt
\noindent
Restricting to fields which are at the center of 
(birational) arithmetic geometry, namely the 
\defi{finitely generated fields} $K$, which are 
the function fields of integral schemes of finite 
type, the~elementary theory $\Th(K)$ is both 
extremely rich and mysterious. The so called
{\it Elementary equivalence vs Isomorphism
Problem\/} (\ha1EEIP\ha1) is about five decades 
old, and asks whether $\Th(K)$ encodes the isomorphy
type of $K$ in the class of all the finitely generated
fields; or equivalently, whether there exists a 
system of axioms in the language of fields~which 
characterizes $K$ among all the finitely generated 
fields. On the other hand, building on \nmnm{Julia} 
\nmnm{Robinson} [Ro1], [Ro2] methods and ideas, 
\nmnm{Rumely}~[Ru] showed at the end of 
the 1970's that for every global field $K$ there 
exists a sentence $\istp K^{\mic{Ru}}$ which 
characterizes the isomorphy type of $K$ as a 
global field, i.e., if $L$ is a global field, then 
$\istp K^{\mic{Ru}}$ holds in $L$ iff $K\cong L$ 
as fields. In other words, the isomorphy type of 
$K$ as a global field is characterized by a 
\textit{single\ha3explicit\ha3axiom 
$\istp K^{\mic{Ru}}$ in the language of fields}.
This goes far beyond the EEIP in the class of global fields!
\vskip2pt
Arguably, it is the \textit{\textbf{main open 
question}\/} in the elementary (or first-order) 
theory of finitely generated fields whether a 
fact similar to \nmnm{Rumely}'s result~[Ru]
holds for all finitely generated fields $K$, 
namely whether there is a field axiom $\istp K$ 
which characterizes the isomorphy type of 
$K$ in the class of all finitely generated fields;
this question is also called the \defi{strong EEIP}. 
We notice that the (strong) EEIP is open in 
general; see~\nmnm{Pop}~[P2],~[P3], for more 
details and references on the EEPI both over
finitely generated fields and function fields over 
algebraically closed base fields. A first attempt 
towards tackling the strong EEIP was
\nmnm{Scanlon}~[Sc], and that reduces the strong EEIP 
for each $K$ to first-order defining ``sufficiently 
many'' divisorial valuations of $K$.\footnote{\ha1See
the discussion below for more about this.}
Finally, \nmnm{Pop}~[P4] tackles the strong EEIP 
for finitely generated fields which are function fields 
of curves over global fields.
In the present note we generalize that result to all 
finitely generated fields of characteristic $\neq2$.
\begin{maintheo*}
For every finitely generated field $K$ with 
$\chr(K)\neq2$, there exists a sentence $\istp K$ in the 
language of fields such that for all finitely generated 
fields $L$ one has: 
\vskip4pt
\centerline{$\istp K$ holds in $L$ if and only if 
$\,L\cong K$ as fields.}
\end{maintheo*}
The Main Theorem above will be proved in Section~5. 
One can give three (by some standards similar) proofs. 
A first-proof follows simply from~\nmnm{Scanlon},
by invoking~Theorem~\ref{TheMThm} below for the 
definability of geometric prime divisors (thus circumventing 
the gap in the proof of defining divisorial valuations in 
Section~3 of loc.cit.). A second proof reduces the Main 
Theorem above to results by \nmnm{Aschenbrenner--%
Kh\'elif--Naziazeno--Scanlon}~[AKNS], by showing 
that finitely generated integrally closed subdomains 
in finitely generated fields of characteristic $\neq2$ are 
uniformly first-order definable. Among other things, 
these proofs show that finitely generated fields of 
characteristic $\neq2$ are \textit{\textbf{bi-interpretable 
with arithmetic}}, see e.g.~[Sc], Section 2, and/or 
[AKNS],~Section~2, for a detailed discussion of 
bi-interpretablility with arithmetic. Third, 
a more direct proof based on \nmnm{Pop} [P2], 
\nmnm{Poonen}~[Po1], and consequences of
\nmnm{Rumely}~[Ru] (namely that the number 
fields are bi-interpretable with arithmetic).
\vskip2pt
The main step and technical key point in the proof 
of the Main Theorem is to give formulae $\val_d$,
all $d>0$, in the language of fields, which uniformly 
first-order define the \defi{geometric prime divisors} 
of finitely generated fields $K$ with $\chr(K)\neq2$ 
and $\dim(K)=d$. That is the content of 
Theorem~\ref{TheMThm} below, which could 
be viewed as the main result of this note.
\vskip5pt
To make these assertions more precise, let us introduce 
notation and mention a few fundamental facts about 
finitely generated fields, to be used throughout the 
manuscript.
\vskip2pt
For arbitrary fields $\Omega$, let 
$\kappa_0\subset\Omega$ denote their prime fields.
Recall that the Kronecker~dimension of $\Omega$ is 
$\dim(\Omega)=\dim(\kappa_0)+\td(\Omega|\kappa_0)$, 
where  $\td(\Omega|\kappa_0)$ 
denotes the transcendence degree, and $\dim(\lvF_p)=0$, 
$\dim(\lvQ)=1$. We denote by $\kappa:=\abs\Omega$ 
the \defi{constant subfield} of~$\Omega$, i.e., 
the elements of $\Omega$ which are algebraic over 
the prime field $\kappa_0\subset\Omega$, and set 
$
\tl\Omega:=\Omega[\sqrt{-1}\ha2].
$
For $\bma:=(a_1,\dts,a_r)$ with $a_i\in\Omega^\times$ 
we consider the $r$-fold Pfister form\footnote{\ha2See 
e.g.\ \nmnm{\footnotesize\sc Pfister}~[Pf1], Ch.\ha{2}2, 
for basic facts.} $q_{\bmam}(\bm x)$ in the variables 
$\bm x=(x_1,\dts,x_{2^r})$ and for field 
extensions $\Omega'|\Omega$ define the \defi{image 
of $\Omega'$ under $q_{\bmam}$} as being
\[
q_{\bmam}(\Omega'):=\{q_{\bmam}(\bm x')\mid
\bm x'\in{\Omega'}^{2^r}\!\!,\,\bm x'\neq\bm0\}.
\]
Next we recall that, 
using among other things the 
Milnor Conjecture,\footnote{\ha2Proved by 
\nmnm{\footnotesize\sc Vojevodsky, Orlov--Vishik--Vojevodsky}, 
and \nmnm{\footnotesize\sc Rost}, see e.g.\ the survey 
articles [Pf2], [Kh].} by \nmnm{Pop}~[P2] there are 
sentences $\varphi_d$, and by \nmnm{Poonen}~[Po1]  
there is a predicate $\abs\psi(\xx)$, and formulas 
$\psi_r(\ttt)$ with free variables~$\ttt:=(\tx_1,\dts,\tx_r)$ 
such that for all finitely generated fields $K$ and 
$\kappa=\abs K\subset K$, setting $\tlK:=K[\sqrt{-1}]$, 
one~has:
\vskip5pt
\itm{20}{
\item[-] $\dim(K)=d\,$ iff $\,\varphi_d$ holds in $K$.
Actually, $\varphi_d\equiv\big((\varphi^0_d\wedge2=0)
\vee(\varphi^0_{d+1}\hb1\wedge2\neq0)\big)$, where
\[
\varphi^0_r \, \equiv \, 
\big(\,\exists\,\bma\hb1=\hb1(a_1,\dts,a_r) \ 
\scalebox{.8}[1]{s.t.} \ 0\not\in q_{\bmam}(\tlK)\big)\,\&\, 
\big(\,\forall\,\bma\!=\!(a_1,\dts,a_{r+1}) \ 
\scalebox{.8}[1]{one\ has} \ha5 0\in q_{\bmam}(\tlK)\big).
\]
\item[-] $\kappa$ is defined by $\abs\psi(\xx)$ inside
$K$, i.e., one has \ $\kappa=\{x\in K\mid\abs\psi(x) 
           \hbox{ holds in } K\}$.
\vskip2pt
\item[-] $t_1,\dts,t_r\in K$ are algebraically 
independent over $\kappa$ \ iff \ $\psi_r(t_1,\dts,t_r)$ 
holds in $K$.
}
In particular, for algebraically independent elements 
$\bmt_r:=(t_1,\dots,t_r)$ of~$K$, the relative algebraic 
closure $k_{\bmtm_r}$ of $\kappa(\bmt_r)$ in $K$ 
is {\it uniformly first-order definable,\/} hence so are 
the maximal global subfields $\ko\subset K$ of $K$,
as well as the transcendence bases $\clT:=(t_1,\dots,t_{d_K})$ 
of $K|\kappa$.
\vskip2pt
A \defi{prime divisor of $K$} is (the 
valuation ring of) a valuation $v$  
whose residue field $\Kv$~satisfies 
\[
\dim(\Kv)=\dim(K)-1.
\] 
It turns out that prime divisors $v$ of finitely 
generated fields are discrete valuations, and 
$\Kv$ is a finitely generated fields as well. 
A prime divisor $v$ of $K$ is called
\defi{arithmetic}, if $v$ is non-trivial on 
$\kappa=\abs K$ ---\ha1in particular $\kappa$ 
must be a number field, respectively 
\defi{geometric}, if $v$ is trivial on $\kappa$. 
Recall that \nmnm{Rumely} [Ru] gives 
formulae $\val_1$ which uniformly first-order 
define the prime divisors~of~global fields, and 
\nmnm{Pop}~[P4] gives formulae $\val_2$ which 
uniformly first-order define the {\it geometric 
prime divisors\/} in the case $\dim(K)=2$. The 
focus of this note is to give similar formulae 
$\val_d$ for fields \lower1pt\hbox{\large$\bullet$} 
satisfying:

\vbox{
\vskip8pt
\centerline{{\bf (H)} \ha{50}\lower1pt
\hbox{\large$\bullet$} \ is finitely generated, \ \ \ 
$d:=\dim(\lower1pt\hbox{\large$\bullet$})>2$, \ \ \ 
$\chr(\lower1pt\hbox{\large$\bullet$})\neq2$\ha{70}}
\begin{theorem}
\label{TheMThm}
There is an explicit procedure that, given an integer 
$d>1$, produces a first-order formula $\,\val_d\,$ 
that in any finitely generated field $K$ of characteristic
$\chr(K)\neq2$ and Kronecker dimension $\dim(K)=d$ 
defines all the geometric prime divisors of $K$.
\end{theorem}
%
%
\noindent 
For the proof see~Section~4, Theorem~\ref{mainthm}, 
and Recipe~\ref{therecipe} for the concrete form 
of~$\,\val_d$. 
\vskip5pt
\noindent
We conclude the Introduction with the following
remarks. }
\vskip2pt
First, in the early version~[P5], the case of 
finitely generated fields of characteristic zero was 
considered\ha1/\ha1dealt with. The methods and
techniques of~[P5] are unchanged, 
except~a~key technical point of the procedure of 
giving $\val_d$, namely the old Proposition~3.5, 
whose new variant Proposition~3.4 below works
for all finitely generated fields satisfying Hypothesis~(H).
\vskip2pt
Second, although the formulae $\val_d$ are completely 
explicit, see Recipe~\ref{therecipe}, it is an open question 
whether these formulae are optimal in any concrete 
sense; in particular, the formulae $\val_d$ do not 
address the question about the \textbf{\textit{complexity 
of (uniform) definability}} of (some or all) the prime 
divisors. The complexity of definability of valuations 
deserves further special attention, because among other 
things it ties in with previous first-order definability 
results of valuations (of finitely generated fields 
and more general fields) by \nmnm{Eisentr\"ager}~[Ei],
\nmnm{Eisentr\"ager--Shlapentokh}~\hbox{[E-S]}, 
\nmnm{Kim-Roush}~[K-R], \nmnm{Koenigsmann}~[Ko1, Ko3],
\nmnm{Miller-Shlapentokh}~[M-Sh], \nmnm{Poonen}~[Po2],
\nmnm{Shlapen-} \nmnm{tokh} [Sh1], [Sh2], and 
others. The focus of the aforementioned results and 
research is yet another open problem in the theory 
of finitely generated fields and function fields, 
namely the generalized Hilbert Tenth Problem 
---\ha2which for the time being is open over all 
number fields, e.g.\ $\lvQ$, $\lvC(t)$, etc.
\vskip2pt
Third, it is strongly believed that the (strong) EEIP
should hold for the function fields $K|k$ over 
``reasonable'' base fields $k$; in particular, since 
finitely generated fields of characteristic $\neq2$
are nothing but function fields $K$ over prime
fields with $\chr\neq2$, the Main Theorem above 
asserts that $\lvQ$ and $\lvF_p$, $p\neq2$, are 
``reasonable.'' If $k$ is an algebraically closed field, facts  
proved by \nmnm{Durr\'e} [Du], \nmnm{Pierce}~[Pi], 
\nmnm{Vidaux}~[Vi] for $\td(K|k)=1$, respectively 
\nmnm{Pop}~[P2] for $\td(K|k)$ arbitrary, are quite 
convincing partial results supporting the possibility that
algebraically closed fields are ``reasonable.'' Finally, 
\nmnm{Koenigsmann}~[Ko2], \nmnm{Poonen-Pop}~[P-P] 
give evidence for the fact that the much more general
\defi{large fields $k$}, as introduced in \nmnm{Pop}~[P1], 
e.g.\ $k=\lvR,\lvQ_p$, PAC, etc., should be ``reasonable'' 
base fields. These partial/preliminary results over large 
fields (including the algebraically closed ones) do not 
involve prime divisors of $K|k$. Two fundamental 
open questions arise: First, is it possible to recover 
prime $k$-divisors of functions fields $K|k$ over 
large fields $k$, at least in the case of special classes 
of large fields, e.g.\ local fields, or quasi-finite fields? 
Second, are there alternative approaches (which do not 
involve prime divisors) for recovering the isomorphy 
type of $K$ from $\Th(K)$? 
\vskip7pt
{\setlength{\baselineskip}{.95\baselineskip}
{\footnotesize
\noindent
{\bf Thanks}: I would like to thank the 
participants at several activities, e.g., AWS 2003, AIM
Workshop 2004, INI Cambridge 2005, HIM Bonn 
2009, ALANT III in 2014, MFO Oberwolfach in 2016, 
IHP Paris in 2018, for the debates on the topic and 
suggestions concerning this problem. Special thanks 
are due to Bjorn Poonen, Thomas Scanlon, Jakob 
Stix, and Michael Temkin for discussing  technical 
aspects of the proofs, 
and to Uwe Jannsen and Moritz Kerz for discussions 
concerning Kato's higher dimensional Hasse local-global 
principles. The author would also like to thank the 
University of Heidelberg and the University of Bonn 
for the excellent working conditions during his visits 
in 2015 and 2016 as a {\it Humbodt Preistr\"ager.\/}}
\par}
%
%
\section{Higher dimensional Hasse local-global principles}
\noindent
A) {\it Notations and general facts\/}
\vskip5pt
For a (possibly trivial) valuation $v$ of $K$, let 
$\eum_v\subset\clO_v\subset K$ be its valuation
ideal and valuation ring ring, $vK:=K\tms/U_v$ 
be its (canonical) value group, and 
$K\hb1v:=\kappa(v):=\clO_v/\eum_v$ be its 
residues field. We denote by $\Val K$ the 
\defi{Riemann space} of $K$, i.e., the space of 
all the (equivalence classes of) valuations of $K$. 
\vskip2pt
Let $X$ be a scheme of finite type over either
$\lvZ$ or a field $k$. For $x\in X$, let 
$X_x:=\oli{\{x\}}\subset X$ be the closure of 
$x$ in $X$, and recall that $\dim(x):=\dim(X_x)$. 
Following \nmnm{Kato}~[Ka], define:
\[
X_i:=\{x\in X\mid\dim(x)=i\}, \quad
X^i:=\{x\in X\mid\codim(x)=i\},
\] 
and recall that if $X$ is projective normal integral, 
then for all $0\leqslant i\leqslant \dim(X)$ one has:
\[
\dim(X)=\codim(x)+\dim(x),
\quad\hbox{and therefore:} \ \
x\in X^i \ \Leftrightarrow \ x\in X_{\dim(X)-i}
\]
\begin{notations/remarks}
\label{nota-1}
Let $K$ be a finitely generated field, and 
$k\subset K$ be a subfield. 
\vskip2pt
\itm{25}{
\item[1)] A \defi{model} of $K$ is a separated 
scheme of finite type $\clX$ with function field 
$\kappa(\clX)=K$. And a \defi{$k$-model} of 
$K$ is a $k$-variety, i.e., a separated $k$-scheme 
of finite type, with $k(X)=K$. 
\vskip2pt
\item[2)] Let a model $\clX$ of $K$, and 
$v\in\Val K$ be given. We say that $v$ has 
\defi{center $x\in\clX$ on $\clX$},  if 
$\clO_x\prec\clO_v$, that is, $\clO_x\subseteq\clO_v$ 
and $\eum_x=\eum_v\cap\clO_x$. By the {\it valuation 
criterion\/} one has:  Since $\clX$ is separated, 
every $v\in\Val K$ has at most one center on $\clX$,
respectively: $\clX$ is proper iff every valuation 
$v\in\Val K$ has a center on $\clX$ (which is 
then unique). 
\vskip2pt
\item[3)] Let a $k$-model $X$ of $K$ and 
$v\in\Val K$ be given. We say that $x\in X$ 
is the center of $v$ on $X$, if $\clO_x\prec\clO_v$. 
If so, then $v\in \Val{K|k}$. By the {\it valuation 
criterion\/} one has: Since $X$ is 
separated over $k$, every $v$ has at most 
one center on $X$, respectively that $X$ is 
a proper $k$-variety iff every $k$-valuation 
$v\in\Val{K|k}$ has a center on $X$ (which
is then unique). 
\vskip2pt
\item[4)] A \defi{prime divisor} of $K$ is any
$v\in\Val K$ satisfying the following equivalent 
conditions:
\vskip2pt
\itm{20}{
\item[i)] $\dim(\Kv)=\dim(K)-1$.
\vskip2pt
\item[ii)] $v$ is discrete, and $\Kv$ is finitely 
generated and has $\dim(\Kv)=\dim(K)-1$.
\vskip2pt
\item[iii)] $v$ is defined by a prime Weil divisor of 
a projective normal model $\clX$ of $K$.
}
\vskip2pt
\item[5)] A \defi{prime $k$-divisor} of $K$ is 
any $v\in\Val{K|k}$ satisfying the following 
equivalent conditions:
\vskip2pt
\itm{20}{
\item[i)] $\td(\Kv\ha1|\ha1k)=\td(K|k)-1$.
\vskip2pt
\item[ii)] $v$ is a prime divisor of $K$
which is trivial on $k$.
\vskip2pt
\item[iii)] $v$ is defined by a prime Weil divisor of
a projective normal model $X$ of $K|k$. 
}
\vskip2pt
\item[6)]
Let $\clD^1_K\supset\clD^1_\clX$ be the spaces 
of prime divisors of $K$, respectively the ones 
defined by the prime Weil divisors of a quasi-projective
normal model $\clX$ of $K$. Further define 
$\clD^1_{K|k}\supset\clD^1_X$ correspondingly,
where $X$ is a quasi-projective normal $k$-model 
of $K$.
\vskip2pt
\item[7)] In the above notation, let $\clX$ and $X$ 
be projective. Then one has canonical identifications:
\vskip5pt
\centerline{$\clD^1_\clX\ \leftrightarrow\ \clX^1\,
=\,\clX_{\dim(K)-1}, \quad\quad \clD^1_X\ 
\leftrightarrow\ X^1\,=\,X_{\td(K|k)-1}$.}
}
\end{notations/remarks}
\vskip-5pt
\noindent
B) {\it Local-global principles\/} (LGP) 
\vskip5pt
Let us first recall the famous Hasse\ha1--Brauer\ha1--Noether 
LGP. Let $k$ be a global field, $\lvP(k)$ be the set of 
non-trivial places of $k$, and for $\vk\in\lvP(k)$, 
let $k_\vk$ be the completion of $k$ with respect 
to~$v$. Denoting by ${}_n(\ha4)$ the $n$-torsion 
in an Abelian group, e.g.\ ${}_n(\lvQ/\lvZ)\cong\lvZ/n$,
the Hasse\ha1--Brauer\ha1--Noether LGP asserts
that one has a canonical exact sequence:
\vskip7pt
\centerline{$0\to{}_n\Br(k)\to\mathop{\oplus}\limits_{\vk}
              {}_n\Br(k_\vk)\to\lvZ/n\to0$,}
\vskip4pt
\noindent
where, the first map is the direct sum of all the canonical
restriction maps ${}_n\Br(k)\to{}_n\Br(k_\vk)$, 
and the second map is the sum of the invariants 
{\small$\sum_v{\rm inv}_v$}.
\vskip3pt
It is a fundamental observation by \nmnm{Kato}~[Ka] 
that the above local-global principle has higher 
dimensional variants as follows: First, following 
\nmnm{Kato}~loc.cit, for every positive integer~$n$, 
say $n=m p^r$ with $p$ the characteristic exponent 
and $(m,p)=1$, an integer twist $i$,  one sets $\lvZ/n(0)=\lvZ/n$, 
and defines in general 
$\lvZ/n(i):=\mu_m^{\otimes i}\oplus W_r\,\Omega^i_{\rm log}[-i]$,
where $W_r\,\Omega_{\rm log}$ is the logarithmic 
part of the de Rham--Witt complex on the \'etale 
site, see \nmnm{Illusie}~[Ill] for details. 
With these notations, for every (finitely generated) 
field $K$ one has:
\vskip3pt
\centerline{$\HHx1K0=
{\rm Hom}_{\rm cont}(G_K,\lvZ/n)$, \ha9
      $\HHx 2 K 1={}_n\Br(K)$,}
\vskip3pt
\noindent
where $G_K$ is the absolute Galois group of $K$. 
Thus the cohomology groups $\HHx{i+1} K i$ 
have a particular arithmetical significance, and in these 
notation, the Hasse\ha1--Brauer\ha1--Noether LGP 
is a local-global principle for the cohomology group 
$\HHx 2 K 1$. Noticing that $K$ is a 
global field iff $\,\dim(K)=1$, \nmnm{Kato} had the
fundamental idea that for finitely generated fields $K$ 
with $\dim(K)=d$, there should exist similar local-global 
principles for $\HHx{d+1} K d$. 
\vskip7pt
%
%
\noindent
$\bullet\,$ {\it The Kato cohomological complex\/} \KC  
\vskip5pt                                             
We briefly recall Kato's cohomological complex
(similar to complexes defined by the Bloch--Ogus) 
which is the basis of the higher dimensional Hasse 
local-global principles, see~\nmnm{Kato}~[Ka],~\S1,
for details. Let $L$ be an arbitrary field, and 
recall the canonical isomorphism (generalizing 
the classical Kummer Theory isomorphism) 
$h^1:L\tms/n\to \HHx 1 L 1$. As explained 
in~[Ka],~\S1, the isomorphism $h^1$ gives 
rise canonically for all $q\neq0$ to morphisms, 
which by the (now proven) Milnor--Bloch--Kato 
Conjecture are actually isomorphisms:
\[
h^q:K^{\mic M}_q(L)/n\to\HHx q L q,\
\{a_1,\dots,a_q\}/n\mapsto \cppr{h^1(a_1)}{h^1(a_q)}
=:\cppr{a_1}{a_q}.
\]
Further, let $v$ be a discrete valuation of $L$. Then
one defines the \defi{boundary homomorphism}
\[
\partial_v:\HHx{q+1}L{q+1}\to\HHx{q}{Lv}{q},
\]
defined by $a\mpstt{\partial_v}\hb1v(a)$ if $q\!=\!0$,  
$a\cpm\cppr{a_1}{a_q}\mpstt{\partial_v}v(a)\cdot\ha1 
                              \cppr{a_1}{a_q}$
for $a\!\in\! L\tms\!$, $a_1,\dts,a_q\!\in\! U_v$~if~$q\!>\!0$.  
\vskip2pt
Now let $X$ be an excellent integral scheme, 
with generic point $\eta_X$, and recall the 
notations $X_i, X^i\subset X$; hence $X_0\subset X$ 
are the closed points, and $X_{\dim(X)}=\{\eta_X\}$. 
By mere definitions, for every $x_{i+1}\in X_{i+1}$, 
one has that $X_{x_{i+1}\hb1,1}\subset X_i$ 
consists of all the points $x_i\in X_i$ which lie 
in the closure of $X_{x_{i+1}}:=\oli{\{x_{i+1}\}}$. 
Since $X$ is excellent,  the normalization 
$\tl X_{x_{i+1}}\to X_{x_{i+1}}$ 
of $X_{x_{i+1}}$ is a finite morphism. Hence 
for every $x_i\in X_{x_{i+1}\hb1,1}$, there a 
finitely many $\tlx\in\tl X_{x_{i+1}}$ 
such that $\tlx\mapsto x_i$ under 
$\tl X_{x_{i+1}}\to X_{x_{i+1}}$,  and the following 
hold: The local rings $\clO_{\tlx}$ of all $\tlx\mapsto x_i$ 
are discrete valuations rings of the residue field 
$\kappa(x_{i+1})$, say with valuation $v_{\tlx}$, 
and the residue field extensions $\kappa(\tlx)|\kappa(x_i)$ 
are finite field extensions. Then for every integer $n>1$,
which is invertible on $X$, letting $0\leqslant i<\dim(X)$, 
one gets a sequence of the form:
\[
\dts\!\to\oplus_{x_{i+1}\in X_{i+1}}
  \HHx{i+2}{\kappa(x_{i+1})}{i\!+\!1}\!\to
    \oplus_{x_i\in X_i}\HHx{i+1}{\kappa(x_i)}{i}
       \!\to\dts
\leqno{\indent{\KC}}   
\]
where the component 
$\HHx{i+2}{\kappa(x_{i+1})}{i+1}\to
\oplus_{x_i\in X_i}\HHx{i+1}{\kappa(x_i)}{i}$
is defined by
\[
{\textstyle\sum}_{\tlx\in \tl X_{x_{i+1}}\!,\,\tlx\mapsto x_i}
{\rm cor}_{\kappa(\tlx)|\kappa(x_i)}\circ\delta_{v_{\tlx}}
\]
\begin{theorem}
[\nmnm{Kato}~\hbox{[Ka]},~Proposition~1.7]
\label{KC}
Suppose that $X$ is an excellent
scheme such that for all $p$ dividing $n$ and 
$x_i\in X_i$ one has: If $p=\chr\big(\kappa(x_i)\big)$,
then $[\kappa(x_i):\kappa(x_i)^p]\leq p^i\!$. Then
\KC\ is a complex. In particular, if $n$ is invertible
on $X$, then \KC\ is a complex.
\end{theorem}
That being said, the \defi{Kato Conjectures} 
are about aspects of the fact that in {\it arithmetically 
significant situations,\/} the complex~\KC\ above is 
exact, excepting maybe for $i=0$, where the homology 
of~\KC\ is perfectly well understood. And \nmnm{Kato}
proved himself several forms of the above local-global 
Principles in the case $X$ is an arithmetic scheme of
dimension $\dim(X)=2$ and having further properties.
Among other things, one has:
\begin{theorem}
[\nmnm{Kato}~\hbox{[Ka]},~Corollary,~p.145]
\label{kato}
Let $X$ be a proper regular integral $\lvZ$-scheme,
$\dim(X)=2$, and $K=\kappa(X)$ having no orderings. 
Then one has an exact sequence:
\[
0\to\HHx3K2\to
\oplus_{x_1\in X_1}\HHx{2}{\kappa(x_1)}{1}
\to\oplus_{x_0\in X_0}\HH^1\big(\kappa(x_0),\lvZ/n\big)
\to\lvZ/n\to0.
\]
\end{theorem}
\vskip2pt
Finally, notice that in~Theorem~\ref{kato} above, 
$K$ is finitely generated with $\dim(K)=2$.
Unfortunately, for the time being, the above result
is not known to hold in the same form in higher 
dimensions $d:=\dim(K)>2$, although it is conjectured
to be so. There are nevertheless partial results 
concerning the local-global principles involving 
$\HHx{d+1}{K}{d\,}$. From those results, we pick 
and choose only what is necessary 
for our goals, see below. 
\begin{notations/remarks}
\label{nota0} 
Let $K$ be a finitely generated field with constant 
field~$\kappa$. We supplement Notations/Remarks 
\ref{nota-1} as follows:
\vskip2pt
\itm{25}{
\item[1)] $n>1$ is a positive integer not divisible 
by $\chr(K)$. 
\vskip2pt
\item[2)] $\ko\subset K$ is {\it global subfield,\/} 
and $\So$ be the canonical model of $\ko$, i.e.,
$\So=\Spec\clO_{\ko}$ if $\ko$ is a number field,
respectively $\So$ is a projective smooth curve 
if $\kappa$ is finite.
\vskip2pt
\item[3)] Let $\lvP_{\mic{fin}}(\ko)$ be 
the set of finite places of $\ko$. 
For $v\in\lvP(\ko)$, consider/denote:
\vskip2pt
\itm{10}{
\item[-] The Henselization $R_\vo$ of $\clO_v$. 
Hence $R_v$ is a Henselian DVR with finite residue field.
\vskip2pt
\item[-] The Henseliazation $\kov=\Quot(R_v)$ of
$\ko$ at~$v$.
}}
\end{notations/remarks}
\vskip5pt
\noindent
$\bullet\,$ {\it Localizing the global field $\ko$\/} 
\vskip5pt
In the above notations, for every $v\in\lvP(\ko)$, 
consider the compositum $K_v:=K\kov$ of $K$ and 
$\kov$ (in some  fixed algebraic closure $\oli K$). 
Then via the restriction functor(s) in cohomology, one 
gets canonical \defi{localization maps} \
$
\HHx{d+1}{K}{d}\to\HHx{d+1}{K_v}{d\,}.
$
%
%
%
%
\begin{theorem}
[\nmnm{Jannsen}~\hbox{[Ja]},~Theorem~0.4]
\label{Jannsen}
In the above notations, suppose that $\chr(K)$ does not
divide~$n$. Then the localization maps give rise 
to an embedding
\[
\HHx{d+1}K{d\,}\to
    \oplus_{v\in\lvP(\ko)}\HHx{d+1}{K_v}{d\,}.
\]
\end{theorem}
\vskip3pt
\noindent
$\bullet\,$ {\it Local-global principles over $R_v$, 
                                             $v\in\lvP(\ko)$\/}
\vskip5pt 
In the above notations, for every non-archimedean
place $v\in\lvP(\ko)$, let $R_v\subset\kov$ be the 
(unique) Henselization of the valuation ring $\clO_v$
inside $\kov$, hence recall that $R_v$ is a Henselian 
discrete valuation ring with residue field $\kappa(v)$ 
finite. This being said, one has the following:                                            
\begin{theorem}[\hbox{\nmnm{Kerz--Saito}~[K--S], 
Theorem~8.1}]
\label{KerzSaito}
Suppose that $R$ is either {\rm(i)}\ha2a finite field, or 
{\rm(ii)}\ha2a~Henselian discrete valuation ring 
with finite residue field, such that $n$ is invertible 
in $R$, and $\mu_n\subset R$. Then for every 
projective regular flat $R$-scheme $X$, the complex 
\KC\ for $X$ is exact, with the only exception 
of the homology group $\,\HH_0\KC=\lvZ/n$ 
in the case~{\rm(i)}. 
\end{theorem}
\section{Consequences/applications of the 
                        local-global principles}
\noindent
In this section we give a few consequences of the 
higher Hasse local-global principles mentioned 
above, as well as an arithmetical interpretation of
these consequences. 
\vskip5pt
\noindent
A) \ {\it A technical result for later use\/}
\vskip5pt
\begin{notations/remarks}
\label{nota1}
Let $L$ be a field satisfying Hypothesis~(H) 
from the Introduction. Let $\ko\subset L$ be 
a global subfield, $n\neq\chr(L)$ be a {\it prime 
number,\/} and suppose that $\mu_{2n}\subset \ko$.
\vskip2pt
\itm{25}{
\item[1)] $\Uo=\Spec\Ro\subset \So$ are open 
subsets with $n\in\Ro^\times$, and set: 
\vskip3pt
\centerline{$\Delta_{\Uo}:=\{\uop\in
  \ko\tms\mid v(\uop-1)>2\cdot v(n),\, 
   \forall\,v\not\in\Uo\}$.}
\vskip3pt
\item[2)] For $\bm f:=(f_1,\dts,f_r)$ with 
$f_i\in L\tms\!$, and dense open subsets 
$\Uo\subset\So$, denote:
\[
\HH_{\Uo,\bm f}:=
\big\langle\uop\cpm\cppr{f_1}{f_r}
\mid\uop\in\Delta_{\Uo}\big\rangle\subset
\HH^{r+1}\big(L,\lvZ/n(r)\big).
\]
}
\end{notations/remarks}
\begin{proposition}
\label{prop-compat} 
In the above notation, let $\clZ$ be a projective 
smooth $\Uo$-variety with generic fiber 
$Z:=\clZ\times_{\Uo}\!\ko$, and function field 
$L=\ko(Z)$ with $\dim(L)=r$. The 
following hold:
\vskip2pt
\itm{25}{
\item[{\rm1)}] The map 
$\,\HH_{\Uo,\bm f}\to\oplus_{z\in Z^1} 
\HH^r\big(\kappa(z),\lvZ/n(r\!-\!1)\big)$
from the Kato complex 
{\rm\defi{(KC)}}~is~injective.
\vskip2pt
\item[{\rm2)}] Let $\bm\alpha=\uop\cpm\cppr{f_1}{f_r}
\in\HH_{\Uo,\bm f}$ and $z\in Z^1$ satisfy
$\partial_z(\bm\alpha)\neq0$, and $w$ be the 
prime divisor of $L|\ko$ with $\clO_w=\clO_z$.
Then there is an $f_i$ such that 
$w(f_i)\not\in n\hb1\cdot\hb1wL$.
}
\end{proposition}
\begin{proof} For $\bm\alpha\in\HH_{\Uo,\bm f}$
non-trivial, proceed as follows:
\vskip2pt
First, by \nmnm{Jannsen}'s 
Theorem~\ref{Jannsen} above, there exists 
$\vo\in\lvP(\ko)$ such that $\bm\alpha$ is 
non-trivial over $L_{\vo}=L{\ko}_{\vo}$.
In particular, $\uop$ is not an $n^{\rm th}$ power in 
${\ko}_{\vo}$, hence $\uop\in\Delta_{\Uo}$.
In particular, letting $R:=\clO^h_{\vo}$ be the 
Henselization of $\clO_{\vo}$, the base change 
$\clZ_{\!R}=\clZ\times_{\Uo}\!R$ is a smooth 
$R$-variety (because $\clZ$ was a smooth 
$\Uo$-variety). Set $\Spec R=\{\eta_0,\eum\}$. 
\vskip2pt
Second, by the \nmnm{Kerz--Saito} 
Theorem~\ref{KerzSaito} above, there are points 
$z_R\in\clZ_{\!R}^1$ such that~one~has: 
$0\neq\bm\alpha_{z_R}\!:=\partial_{z_R}(\bm\alpha)
\!\in\!\HH^r\big(\kappa(z_R),\lvZ/n(r-\hb11)\big)$.
Hence setting $k_R:=\Quot(R)={\ko}_{\vo}$, 
one has: If $z_R\mapsto\eta_0$ under 
$\clZ_{\!R}\to\Spec R$, then $z_R$ lies in the 
generic fiber $Z_{k_R}=\clZ_{\!R}\times_R{k_R}$ 
of~$\clZ_{\!R}$. Hence letting $z_R\mapsto z$ 
under $\clZ_{\!R}\to\clZ$, one gets: Since
$Z_{k_R}=Z\times_{\ko}k_R$, one has
$z\in Z^1\hb3$. Second, since 
$\kappa(z)\hra\kappa(z_R)$, it follows that
$0\neq\bm\alpha_z\!:=\partial_z(\bm\alpha)
\in\HH^r\big(\kappa(z),\lvZ/n(r-\hb11)\big)$,~as~claimed.
Next suppose that $z_R\mapsto\eum$ under 
$\clZ_{\!R}\to\Spec R$. Since $\clZ_{\!R}$ 
is a projective smooth $R$-variety, its special 
fiber $\clZ_\eum$ is reduced and has projective 
smooth integral $\kappa(\eum)$-varieties as
connected components, $\clZ_{z_R}=\oli{\{z_R\}}$
being such one. Since 
$0\neq\bm\alpha_z:=\partial_z(\bm\alpha)
               \in\HHx r{\kappa(z)}{r\hb1-\!1}$,
by \nmnm{Kerz--Saito}'s~Thm~\ref{KerzSaito}, 
there is $y\in\clZ^1_z$ with
$0\neq\partial_{y}\big(\partial_{z_R}(\bm\alpha)\big)
\in\HHx{r-1}{\kappa(y)}{r\hb1-\hb12}$. 
On the other hand, 
$\dim(\clZ_{z_R})=\dim(\clZ_{\!R})\!-\!1$, 
hence $\clZ^1_{z_R}\!\subset\clZ^2_{\!R}$. 
Hence $\codim_{\clZ_{\!R}}(y)=2$, and
$y\mapsto \eum$ under $\clZ_{\!R}\to\Spec R$.
Let $\clZ(y):=\big\{z'_R\in\clZ^1_{\!R}\mid
y\in\oli{\{z'_R\}\!}\,\big\}$.~The~Kato~complex~\KC~for 
the projective smooth $R$-scheme $\clZ_{\!R}$ 
implies: ${\textstyle\sum}_{z'_R\in\clZ(y)}\,
\partial_{y}\big(\partial_{z'_R}(\bm\alpha)\big)=0$.
Since $z_R\in \clZ(y)$ and 
$\partial_{y}\big(\partial_{z_R}(\bm\alpha)\big)\neq0$,
there must exist points $z'_R\in \clZ(y)$
satisfying: 
\[
z'_R\neq z_R,\quad
0\neq\partial_{y}\big(\partial_{z'_R}(\bm\alpha)\big)
\in\HH^{r-1}\big(\kappa(y),\lvZ/n(r\hb1-\hb{1}2)\big).
\]
We claim that all $z'_R\in \clZ(y)$, $z'_R\neq z_R$ 
must~satisfy: $z'_R\mapsto\eta_0\in\Spec R$ under
$\clZ_{\!R}\to\Spec R$. By contradiction, suppose that 
$z'_R\mapsto\eum$, or equivalently, $z'_R\in\clZ_\eum$.  
Arguing as above about as we did for~$z_R$, it follows that
$\clZ_{z'_R}=\oli{\{z'_R\}}$ is a connected 
component of $\clZ_\eum$. Since the connected
components $\clZ_{z_R}$
and $\clZ_{z'_R}$ are either identical or disjoint,
and $y\in\clZ_{z_R}\cap\clZ_{z'_R}$, it
follows that $z'_R=z_R$, contradiction!  
\end{proof}
\begin{notations/remarks}
\label{nota2}
In Notations/Remarks~\ref{nota0}, let $k\subset K$
be a (relatively algebraically closed) subfield 
with $\td(K|k)=1$ and $\ko\subset k$. 
Then there exists a unique projective normal (or
equivalently regular) $k$-curve $C$ such that $K=k(C)$. The 
closed points $P\in C$ are in canonical bijection 
with the prime divisors $v$ of $K|k$ via $\clO_P=\clO_v$. 
\vskip2pt
Let $f\in K\backslash k$ be given.  Our aim 
in this subsection is to give a criterion\ha2---\ha2which 
for $n=2$ and $\chr\neq2$ turns out to be 
{\bf first-order}\ha2---\ha2to express the following:
\vskip5pt
\centerline{\it The set $\DDf:=\{P\in C\,|\, 
v_P(f)\not\in n\cdot v_P(K)\}$ is non-empty.}
\vskip5pt
\noindent
In order to do so, we supplement the previous 
notations and remarks as follows:
\vskip2pt
\itm{25}{
\item[1)] Let $\bmt:=(t_1,\dts,t_e)$ denote 
transcendence bases of $k|\ko$ such that 
$\ti$ are $n^{\rm th}$ powers in $K$.
\vskip2pt
\item[2)] $\lvAi\subset\lvPi$ are the 
$\So$-affine/projective $\ti$-lines, and set
$\lvAt:=\times_i\,\lvAi\subset\times_i\,\lvPi=:\lvPt$.
\vskip2pt
\item[3)] For $\bma=(a_1,\dots, a_e)\in\ko^e$, 
set $\bmu=(u_i)_i=(\ti-a_i)_i$, and for $\Uo,\bma,f$ 
as above, consider:
\[
\HH_{\Uaf}:=
\big\langle u'_0\cpm u_0\cpm \cppr{u_1}{u_e}\cpm f
  \mid u'_0\in\Delta_{\Uo},u_0\in\ko^\times\big\rangle
\subset\HH^{d+1}\big(K,\lvZ/n(d\,)\big).
\]
}
\end{notations/remarks}
\begin{proposition}
\label{prop2} 
In the above notations, the following are equivalent:
\vskip2pt
\itm{25}{
\item[{\rm i)}] $\DDf$ is non-empty.
\vskip2pt
\item[{\rm ii)}] $\exists\,\Ut\subset\lvAt$ 
Zariski open dense such that $\forall\,\bma\in \Ut(\ko)$, 
$\forall\,\Uo$ one has $\,\HH_\Uaf\neq0$.
\vskip2pt
\item[{\rm iii)}] $\forall\,\Ut\subset\lvAt$ 
Zariski open dense $\exists\,\bma\in \Ut(\ko)$ 
such that $\forall\,\Uo$ one has $\,\HH_\Uaf\neq0$.
}
\end{proposition}
\begin{proof} 
To i) $\Rightarrow$ ii): The proof of this 
implication is ``easy'' and requires just standard 
facts. Let $X\to\lvPft$ be the normalization of
$\lvPft:=\lvPf\times\lvPt$ in the field extension 
$\ko(f\!,\bmt)\hra K$. Let $P\in\DDf$ 
be given, hence by mere definitions we have 
$v_P(f)=m\in\lvZ$, and $m$ is not divisible 
by $n$ in $v_PK$. Then the residue map 
$\partial_P:\HH^{d+1}\big(K,\lvZ/n(d)\big)
\to\HH^d\big(\kappa(P),\lvZ/n(d-1)\big)$
restricted to $\HH_\Uaf$ is 
$\partial_P(\uop\cpm \uo\cpm \cppr{u_1}{u_e}\cpm f)=
(-1)^d\ha1m\hb1\cdot\hb1\uop\cpm \uo\cpm \cppr{u_1}{u_e}$.
\vskip2pt
Let $\ko(\bmt)\hra l$ be the separable part of
$\ko(\bmt)\hra\kappa(P)$, and $S_P\to S\to\lvPt$ 
be the normalization of $\lvPt$ in the finite field 
extensions $\kappa(P)\hla l\hla \ko(\bmt)$. Then
$S\to\lvPt$ is \'etale above a dense open subset 
$\Ut\subset\lvAt$. \hbox{Since $\lvAt$ is regular, 
the preimages~$s_{\bmam}\in S$} of 
$\bma=(a_1,\dts,a_e)\in\Ut(\ko)$ are regular 
points, $\bmu:=(u_1,\dts,u_e)=(t_1-a_1,\dts,t_e-a_e)$ is a 
system of regular parameters at $s_{\bmam}$, and  
$\kappa_{\bmam}:=\kappa(s_{\bmam})$ is a finite 
separable extension of $\ko$. Hence letting 
$R,\eum$ be the local ring $\clO_{\!s_{\bmam}},
\eum_{s_{\bmam}}$ of $s_{\bmam}$,  
the $\eum$-adic completion of $R$ is nothing 
but $\wh R=\kappa_{\bmam}[[u_1,\dots,u_e]]$, 
which obviously embeds into
$\wh l:=\kappa_{\bmam}\lps{u_1}\dots\lps{u_e}$.
Since $\ko\hra\kappa_{\bmam}$ is an extension of 
global fields, the image of
$\HH_{\Uo}=\{ \uop\cpm \uo\mid 
(\uop,\uo)\in\Delta_{\Uo}\}$ under
$\,{\rm res}:\HH^2\big(\ko,\lvZ/n(1)\big)\to
\HH^2\big(\kappa_{\bmam},\lvZ/n(1)\big)$
is non-trivial. And if $\uop\cpm \uo$ is non-trivial
over $\kappa_{\bmam}$, and easy induction on $e$ 
shows that $\bm\beta:=\uop\cpm \uo\cpm 
\cppr{u_1}{u_e}$ is non-trivial over $\wh l$, 
thus over $l\subset\wh l$. Finally, since 
$\kappa(P)\,|\,l$ is purely inseparable, $\bm\beta$ 
is non-trivial over $\kappa(P)$, hence  
$0\neq\bm\beta\ha1\cpm\hb1f=\uop\cpm \uo\cpm 
\cppr{u_1}{u_e}\cpm\hb1f\in\HH_\Uaf$.
\vskip5pt
To ii) $\Rightarrow$ iii): Clear!
\vskip5pt
To iii) $\Rightarrow$ i): The (quite involved)  
proof is by induction on $e$.
For every non-empty subset $I\subset\II e:=\{1,\dts,e\}$, 
set $\bmtI=(t_i)_{i\in I}$, $k_I:=\ko(\bmt_I)$, 
and denote $\lvA_I:=\times_{i\in I}\lvAi
\subset\times_{i\in I}\lvPi=:\lvP_I$. For 
$J\subset I\subset\II e$, the canonical projections 
$\lvP_I\srjr\lvP_J$ are open surjective and 
define $k_J\hra k_I$.
\vskip5pt
\noindent
\defi{Construction$\,_\clI$}. Let $\clI=(\II\nu)_\nu$ 
be a chain of subsets of $\II e$ with $|\II\nu|=\nu$,
$\II0=\emptyset$. Setting $k_\nu:=\kIx\nu$, the 
canonical projections 
$\,\lvPt=\lvPIx e\srjr\dts\srjr\lvPIx1$ define  
$\ko(\bmt)=k_e\hla\dts\hla k_1$. Given 
Zariski open dense subsets $\Unu:=\UIx\nu\subset
\lvPIx\nu$ with $\,\Uu e\to\dts\to\Uu1$, and  
 the preimages $\lvPfU\nu\hb3\srjr\lvPU\nu$ 
of $\Unu$ under $\lvPft\srjr\lvPt\srjr\lvPIx\nu$, 
one has canonical open $\ko$-immersions:
\vskip5pt
\centerline{$(*)_\clI\ha{25}$ $\Uu e=\lvPU e\!\!
\hra\dts\hra\lvPU1\hb3\hra\lvPIx e\hb2=\lvPt\,$, \ \ \ 
$\lvPfU e\hb3\hra\dts\hra\lvPfU1\hb3\hra
\lvPft\!=\lvPf\!\times\lvPt\,$. $\ha{20}$}
\vskip5pt
Let $X_0\to\lvPft$ be any projective morphism of 
$\ko$-varieties defining $\ko(f\hb2,\hb1\bmt)\hra K$.
Using \defi{prime to $n$ alterations}, see [ILO], 
Expos\'e X, Theorem~2.1, there is a projective smooth 
$\ko$-variety $\tlX_0$ and a projective surjective 
$\ko$-morphism $\tlX_0\to X_0$ defining a field 
extension $K=:K_0:=\ko(X_0)\hra\ko(\tlX_0)=:\tlK_0$ 
of degree prime to $n$. Proceed as follows: 
\vskip2pt
\underbar{Step 1}. \defi{Generic part$\,_\clI$}.
Let $X_1\to\lvP_{k_1}$ be the generic fiber 
of the $\lvP_{\II1}$-morphisms $\tlX_0\to\lvPt$. 
Then $X_1$ is~a~projective regular $k_1$-variety, 
and $K_1\!:=k_1(X_1)=\tlK_0$. By [ILO], 
loc.cit., there is a projective smooth $k_1$-variety 
$\tlX_1$ and a projective $k_1$-morphism 
$\tlX_1\to X_1$ defining a field~extension 
$K_1=k_1(X_1)\hra k_1(\tlX_1)=:\!\tlK_1$ 
of degree prime to $n$. Inductively on~$\nu$, 
for the already constructed projective 
smooth $k_\nuo$-variety $\tlX_\nuo$, let 
$X_\nu\!\to\lvP_{k_\nu}$~be~the~generic 
fiber of the $\lvP_{I_\nu}$-morphism
$\tlX_\nuo\hb2\to\!\lvP_{k_\nuo}$, thus 
$X_\nu$ is a projective~regular~$k_\nu$-variety.~%
By~[ILO], loc.cit., there is a projective smooth 
$k_\nu$-variety $\tlX_\nu$ and a projective 
$k_\nu$-morphism $\tlX_\nu\to X_\nu$ 
such that the field extension $K_\nu=k_\nu(X_\nu)\hra 
k_\nu(\tlX_\nu)=:\tlK_\nu$ has degree prime to $n$.
\vskip2pt
Notice that the generic fibers
$\tlC_\nu\to C_\nu\to\lvP_{\!f\!,k_e}$ of 
the $\lvP_{k_\nu}$-morphisms 
$\tlX_\nu\to X_\nu\to\lvP_{f\!,k_\nu}$ 
are finite morphisms of projective regular 
$k_e$-curves, hence flat morphisms.
\vskip4pt
\underbar{Step 2}. \defi{Deformation part$\,_\clI$}.
Since being a projective smooth and/or
finite flat morphism is an open condition 
on the base, there are Zariski dense open subsets 
$\Unu\subset\lvA_{\II\nu}\subset\lvP_{\II\nu}$
such that~$(*)_\clI$ above holds, 
there are {\it projective smooth $\Unu$-varieties\/}
$\tlclX_\nu$, $1\leqslant\nu\leqslant e$, such that
$\tlclX_\nu\!\to \Unu$ has as generic fiber
the projective smooth $k_\nu$-variety $\tlX_\nu$.
And setting $\tlclX_0:=\tlX_0$, and 
$\clX_\nu\hb2:=\tlclX_\nuo\!\times_{\lvP_{\II e}}
\hb2\lvP_{\Unu}$ for $1\leqslant \nu\leqslant e$, 
there is a projective morphism of $\Unu$-varieties 
$\,\tlclX_\nu\to\clX_\nu$ with generic fiber 
$\tlX_\nu\!\to\! X_\nu$, thus defining the field
embedding $\tlK_\nu\hla K_\nu$ of degree prime to 
$n$ degree. Note that~\defi{Construction}$\ha1_\clI$ 
realizes the morphisms above in dependence on 
$\clI=(\II\nu)_\nu$, that is, one should rather speak 
about $\tlclX_{\clI,\nu}\to\clX_{\clI,\nu}\to U_{\clI,\nu}$, 
etc. One the other hand, for all sufficiently small 
Zariski open dense subsets $U\subset\lvA_{\II e}$ 
the following is satisfied:
\vskip5pt
\centerline{$\ha7(\dag)\ha{10}$ {\it For all
$\,\clI=(\II\nu)_{1\leqslant\nu\leqslant e}$ one has: 
$U\subset U_{\clI,e}$ and $\tlclX_{\clI,\nu}
\hb2\to\!\clX_{\clI,\nu}\hb2\to\!\lvP_{U_{\clI,\nu}}$ 
are flat above $U$.$\ha{10}$}}
\vskip5pt
\noindent
From now on $\Ut:=U\subset\lvA_{\II e}$ always satisfies 
condition $(\dag)$, and we replace $U_{\clI, \nu}$ 
by the image $\Unu:=U_{\II\nu}$ of $\,U$ under 
$\lvP_{\II e}\to\lvP_{\II\nu}$, etc. Hence the objects 
above satisfy the following:
\begin{hypot}
\label{hypo}
$\Ut\subset\lvA_{I_e}$ is a Zariski open subset 
such that for all $\clI=(\II\nu)_{1\leqslant\nu\leqslant e}$, 
and the resulting open surjective projections
$U=U_e\srjr\dts\srjr U_1$, and the open 
immersions $\Uu{}=\lvPU{}\hra\dts\hra
\lvPU1\hb3\hra\lvPIx e\hb2=\lvPt$ and 
$\lvPfU{}\hra\dts\hra\lvPfU1\hb3\hra\hb1
\lvPft\!=\lvPf\!\times\lvPt$, there are\ha1/\ha1one has:
\itm{25}{
\item[1)] Projective smooth $U_\nu$-varieties 
$\tlclX_\nu$ generic fiber the projective smooth 
$k_\nu$-variety $\tlX_\nu$.
\vskip2pt
\item[2)] Projective morphisms of $\Unu$-varieties 
$\,\tlclX_\nu\to\clX_\nu$ 
defining the field embedding $\tlK_\nu\hla K_\nu$, 
where $\clX_\nu\hb2:=\tlclX_\nuo\!
\times_{\lvP_{\II e}}\hb2\lvP_{\Unu}$, and
$\tlclX_0:=\tlX_0$, thus a canonical morphism
$\tlclX_\nu\to\tlclX_{\nu-1}$.
\vskip2pt
\item[3)] Setting $\tlclX_{\nu,\ha1U}:=
\tlclX_\nu\hb2\times_{\lvP_{\Unu}}\hb3U\!$, 
one has projective flat $U$-morphisms of regular 
$U$-curves:
\vskip2pt
\centerline{ 
{\it $\tlclX_{e,U}\to\dts\to\tlclX_{1,U}\to
\tlclX_{0,U}\to\lvP_{\!f\!,U}\,$ defining \ 
$\tlK_e\hla\dts\hla\tlK_1\hla\tlK_0\hla\ko(f\!,\bmt)$.}}
}
\end{hypot}
\noindent
In particular, for all $\clI$ as above, and  
$\bma:=(a_1,\dts,a_e)\in U(\ko)$,
$\bma_\nu:=(a_i)_{i\in \II\nu}\in \Unu(\ko)$,
the fibers $\tlclX_{\bmam_\nu}$ of $\tlclX_\nu\to \Unu$ 
at $\bma_\nu\in\Uu\nu$ is a projective smooth 
$\ko$-varieties. Hence if $\Uo\subset\So$ is a 
sufficiently small Zariski open dense subset
(depending on $\bma$), one has:
\vskip5pt
\centerline{$(\ddag)_{\bmam}$ \ {\it For all 
$\,\clI=(\II\nu)_\nu$, $\tlclX_0$ and all the 
$\ko$-varieties $\tlclX_{\bmam_\nu}$ have 
projective smooth $\Uo$-models.\/}}
\vskip7pt
\noindent
Returning to the proof of implication iii) 
$\Rightarrow$ i), we proceed by proving:
\begin{lemma}
\label{basic-lemma}
Under Hypothesis~\ref{hypo}, let $\bma\in\Ut(\ko)$,
and $\Uo$ satisfy $(\ddag)_{\bmam}$. Then for each  
$\,0\neq\bm\alpha\in\HH_\Uaf$, there is $P\in C$ 
such that $\partial_P(\bm\alpha)$ is non-trivial, and 
$v_P(f)\not\in n\!\cdot\! v_P(K)$. 
\end{lemma}
\noindent
{\it Proof of Lemma~\ref{basic-lemma}.\/} The
proof is by induction on $e$ as follows: Since 
$[\tlK_0:K]$ is prime to $n$, the restriction 
$\res\!:\!\HH^{d+1}\big(K,\lvZ/n(d\ha1)\big)\to
 \HH^{d+1}\big(\tlK_0,\lvZ/n(d\ha1)\big)$ is injective, 
hence $\tl{\bm\alpha}:=\res(\bm\alpha)$ 
is nontrivial over $\tlK_0$. Since $\tlclX_0$ has
a projective smooth $\Uo$-model, by 
Proposition~\ref{prop-compat} there is a 
point $x_0\in\tlclX_0$ with $\codim_{\tlclX_0}(x_0)=1$
and $\bm\beta:=\partial_{x_0}(\tl{\bm\alpha})$ is
nontrivial in $\HH^d\big(\kappa(x_0),\lvZ/n(d\hb1-\!1)\big)$.
Hence the prime divisor $w_0:=w_{x_0}$ of $\tlK_0|\ko$
satisfies:
\vskip5pt
\underbar{Case 1}. $w_0(u_i)=0$ for all $i=1,\dts,e$. 
Then setting $u_i\mapsto\oli u_i$ under 
$\clO_{w_0}=\clO_{x_0}\to\kappa(x_0)$, it follows that
$\partial_{x_0}(\tl{\bm\alpha})=u'_0\cpm u_0\cpm 
\cppr{\oli u_1}{\oli u_e}\neq0$ in 
$\HH^d\big(\kappa(x_0),\lvZ/n(d-1)\big)$. Hence 
$\oli u_1,\dts,\oli u_e$ must be algebraically
independent over $\ko$, or equivalently, $w_0$
must be trivial on $\ko(u_1,\dts,u_e)$. Thus
$w_0|_K=v_P$ for some $P\in C$ such that
$\partial_P(\bm\alpha)\neq$ over $\kappa(P)$,
and $v_P(f)\not\in n\hb1\cdot\hb1v_PK$ (because 
$w(f)\not\in n\hb1\cdot\hb1w\tlK_0$). Thus 
finally implying that $\DDf$ is non-empty, 
as claimed.
\vskip5pt
\underbar{Case 2}. The set 
$I:=\{i\,|\,w_0(u_i)\neq0\}$ is non-empty. 
Recalling that $t_i$ is an $n^{\rm th}$ power
in~$K$, say $t_i=t^n$, and $u_i=\ti-a_i$, 
we claim that $w_0(\ti)=0$. Indeed, by 
contradiction, let $w_0(\ti)\neq0$. First, if
$w_0(\ti)>0$, then $\oli u_i=a_i$, hence 
$\uop\cpm\uo\cpm a_i
\in\HH^3\big(\ko,\lvZ/n(2)\big)=0$ is a 
sub-symbol of $\partial_{x_0}(\tl{\bm\alpha})$,
implying that $\partial_{x_0}(\tl{\bm\alpha})=0$, 
contradiction! Second, if $w_0(\ti)<0$, then $\ti-a_i=
\ti(1-a_i/\ti)=t^n u'_i$ with $u'_i\in1+\eum_{w_0}$. 
Hence $\oli u'_i=1$, and $\bm\alpha=\bm\alpha'$,
where the latter symbol is obtained by replacing
$u_i$ by $u'_i$ in the symbol $\bm\alpha$. 
Then $\uop\cpm\uo\cpm1=0$ is a sub-symbol 
of $\partial_{x_0}(\tl{\bm\alpha})$, thus 
$\partial_{x_0}(\tl{\bm\alpha})=0$, contradiction! 
Thus conclude that $w_0(\ti)=0$, hence
$w_0(t_i-a_i)\neq0\,\Rightarrow\,w_0(t_i-a_i)>0$. 
In particular, replacing $f$ by $f^{-1}$ if 
necessary, w.l.o.g., $w_0(f)\geqslant0$. Thus 
finally $u_1,\dts,u_e,f\in\clO_{x_0}$, and 
$u_i\in\eum_{x_0}$ iff $i\in I$.
\vskip2pt
Consider the images
$x_0\mapsto x_{f\!,\bmam_I}\mapsto x_{\bmam_I}
\mapsto\bma_I:=(a_i)_{i\in I}$ under 
$\tlclX_0\to\lvPft\to\lvP_{\II e}\to\lvP_I$.
Then $\bma_I\in \lvP_{U_I}(\ko)$, and 
$x_{\bmam_I}\in\lvPt$ is defined by $\bmt_I=\bma_I$,
hence $\td\big(\kappa(x_{\bmam_I})|\ko\big)\leqslant e-|I|$.
Further, $x_{\bmam_I}\in U$ (because $x_{\bmam_I}$
is a generalization of $\bma$).
Hence $e=\td\big(\kappa(x_0)|\ko\big)$, 
together with $\tlclX_0\to\lvPft\to\lvPt$ being flat 
at $x_0\mapsto x_{f\!,\bmam_I}\mapsto 
x_{\bmam_I}\in U_e$, imply:
\vskip5pt
\centerline{$e-1=\td\big(\kappa(x_0)|\ko\big)-1=
\td\big(\kappa(x_{f\!,\bmam_I})|\ko\big)-1=
\td\big(\kappa(x_{\bmam_I})|\ko\big)
\leqslant e-|I|$, \ hence $|I|=1$,}
\vskip5pt
\noindent
and $\td\big(\kappa(x_{f\!,\bmam_I})|\ko\big)-1=
\td\big(\kappa(x_{\bmam_I})|\ko\big)$ implies
$f\notin\eum_{x_{f\!,\bmam_I}}\hb3\subset\eum_{x_0}$, 
thus $f\!,u_i\in\clO_{x_0}^\times$, $i\notin I$.
Reasoning as in~Case~1), the non-triviality
of $\partial_{x_0}(\tl{\bm\alpha})$ implies that 
the residues $\oli f\!, \oli u_i\in\kappa(x_0)$ 
of $f\!, u_i$, $i\notin I$ are algebraically 
independent over $\ko$. Equivalently, 
the residues $\oli f\!,\oli{\ha1t\ha1}\hb1_i$, 
$i\notin I$ are algebraically independent over 
$\ko$, hence $x_{\!f\!,\bmam_I}\in\lvPft$ is the 
generic point of the fiber of $\,\lvP_{\!f\!,\bmtm}
\to\lvP_I$ at $\bma_I\in\lvP_I(\ko)$. Since 
$x_0\!\mapsto\! x_{f\!,\bmam_I}\hb2\mapsto\! 
x_{\bmam_I}\hb2\mapsto\bma_I$ under 
$\tlclX_0\!\to\lvPft\!\to\!\lvP_{\II e}\!\to\lvP_I$, 
one has:
\vskip5pt
\centerline{$(*)\ha{20}$ {\it $x_0$ is a generic 
point of the fiber $\tlclX_{0,\bmam_I}$ of 
$\tlclX_0\to\lvP_I$ at $\bma_I\in U_I(\ko)$.$\ha{30}$}} 
\vskip5pt
After renumbering $(t_1,\dts,t_e)$, 
w.l.o.g., $I=\{e\}$. Considering all the chains 
$\clI=(I_\nu)_\nu$ with $I_1=\{e\}$, and 
viewing $\tlclX_\nu\hb2\to\lvP_{\Unu}$ and 
$\tlclX_\nu\hb2\to\clX_\nu\hb2
\to\lvP_{\!f\!,\Unu}$ as $U_1$-morphisms via 
$\Unu\to U_1$, let 
$\tlclX_{\nu,\bmam_1}\to \clX_{\nu,\bmam_1}
\to \lvP_{U_\nu,\bmam_1}\to U_{\nu,\bmam_1}$ 
be the fibers of $\tlclX_\nu\to\clX_\nu\to
\lvP_{U_\nu}\to U_\nu$ at $\bma_1\in U_1(\ko)$.  
Let $w_0$ is the prime divisor of $\tlK_0$ defined 
by $x_0\in\tlclX_0$. Then Hypothesis~\ref{hypo},~3)
implies: The sequences  
$(w_\nu)_{0\leqslant\nu\leqslant e}$ of 
prolongations of $w_0$ to the tower
of extensions $(\tlK_\nu)_\nu$ satisfying
$w_{\nu+1}|_{\tlK_\nu}=w_\nu$ 
are~in~canonical bijection with the sequences 
$(x_\nu)_{0\leqslant\nu\leqslant e}$ of 
generic points $x_\nu\in\tlclX_{\nu,\bmam_1}$ with 
$x_{\nu+1}\mapsto x_\nu$ for all $0\leqslant\nu<e$.
Hence denoting $e_{\nu+1}:=e(w_{\nu+1}|w_\nu)$, 
$f_{\nu+1}:=f(w_{\nu+1}|w_\nu)=
[\kappa(x_{\nu+1}):\kappa(x_\nu)]$ the
ramification index, respectively the residue 
degree of $w_{\nu+1}|w_\nu$, one has: Since
$[\tlK_{\nu+1}:\tlK_\nu]$ is prime to $n$, by the 
fundamental equality, for every $w_\nu$ there 
is a prolongation $w_{\nu+1}$ with 
$e_{\nu+1}f_{\nu+1}$ prime to~$n$. We 
will call such sequences $(x_\nu)_\nu$ and/or 
$(w_\nu)_\nu$ \defi{prime to $n$ compatible}.
Notice that letting $\tlclX_{x_\nu}\subset
\tlclX_\nu$ be the Zariski closure of $x_\nu$,
the morphism $\tlclX_{\nu+1}\to\tlclX_\nu$
gives rise to a morphism $\tlclX_{x_{\nu+1}}
\to\tlclX_{x_\nu}$ defining the finite extension
$\tlL_\nu:=\kappa(x_\nu)\hra\kappa(x_{\nu+1})
=:\tlL_{\nu+1}$ of degree $f_{\nu+1}$ prime 
to $n$. Further, $\tlclX_{x_1}$ is an irreducible 
component of the projective smooth 
$\ko$-variety~$\tlclX_{\bmam_1}=\tlclX_{1,\bmam_1}$, 
hence $\tlclX_{x_1}$ is itself a projective smooth 
$\ko$-variety. And by assumption~$(\ddag)_{\bmam}$,
$\tlclX_{x_1}$ has a projective smooth $\Uo$-model. 
Finally, since $\tlL_0=\kappa(x_0)\hra\kappa(x_1)=\tlL_1$ 
has degree prime to $n$, and 
$\bm\beta=\partial_{x_0}(\tl{\bm\alpha})$
is non-trivial over $\tlL_0=\kappa(x_0)$, it follows that 
$\bm\beta_1=\res(\bm\beta)$ is non-trivial over 
$\tlL_1$. Hence by Proposition~\ref{prop-compat}, 
there is a point $y\in\tlclX_{x_1}^1$ such that 
$\partial_y(\bm\beta_1)\neq0$. Let $\tlclX_1(y)
\subset\tlclX_1^1$ be the set of points $x'_1\in\tlclX_1$
with $\codim_{\tlclX_1}(x'_1)=1$ and 
$y\in\oli{\{x'_1\}}$. Then reasoning as in the 
proof of Proposition~\ref{prop-compat},
it follows that there is $x'_1\in\tlclX_1(y)$ 
such that, first, $\partial_{x'_1}(\tl{\bm\alpha})\neq0$ 
over $\kappa(x'_1)$, and second, $x'_1$ is not
contained in any fiber $\tlclX_{\bmam'_1}$
of $\tlclX_1\to U_1$. Equivalently, $x'_1$ lies 
in the generic fiber $X_1$ of $\tlclX_1\to U_1$, 
and the prime divisor $w'_1$ of $\tlK_1|\ko$ 
defined by $x'_1$ is trivial on~$k_1$.
\vskip5pt
\underbar{Case} $e=1$: \ One has
$\bmt=(t_1)$, hence $k_1=\ko(\bmt)$.
Therefore $w'_1|_K=v_P$ for some $P\in C$, 
and conclude by reasoning as in Case~1 above.
\vskip5pt
\underbar{Case} $e>1$: \ Setting $\mu\hb2:=\nu-1$,
$\JJ\mu\hb2:=I_\nu\backslash I_1$, the inclusions
$\Vu\mu\hb2:=U_{\nu,\bmam_1}\hb2 \subset
\lvP_{\nu,\bmam_1}\hb2=\lvP_{\JJ\mu}$
are~open immersions, where $V_0:=\Spec\ko=:\lvP_{\!\JJ0}$. 
Given a prime to $n$ compatible sequence
$(x_\nu)_{0\leqslant\nu\leqslant e}$, for every
$\nu>0$ we set: $\tlclY_\mu:=\tlclX_{x_\nu}
\subset\tlclX_{\nu,\bmam_1}$, and recall that
$\tlL_\mu=\ko(\tlclY_\mu)=\kappa(x_\nu)$. 
Since $\tlclX_\nu\to U_\nu$ is projective smooth, 
its fiber $\tlclX_{\nu,\bmam_1}\to U_{\nu,\bmam_1}$ 
is projective smooth, hence so is its irreducible 
component $\tlclY_\mu\subset\tlclX_{\nu,\bmam_1}$. 
Further, since $\tlclX_\nu\to\clX_\nu=
\tlclX_{\nu-1}\times_{\lvPt}\lvP_{U_\nu}$ is
a projective $U_\nu$-morphism, the transitivity
of base change implies that 
$\clY_\mu:=\clX_{\nu,\bmam_1}=\tlclY_{\mu-1}
\times_{\lvPtp}\lvP_{V_\mu}$, and the resulting
$V_\mu$-morphism $\tlclY_\mu\to\clY_\mu$ is
projective defining $\tlL_\mu\hla\tlL_{\mu-1}$.
Therefore Hypothesis~\ref{hypo} holds 
{\it mutatis mutandis\/} in the context below, 
after replacing $e$ by $e'\!$, namely:
\vskip5pt
For the canonical open projections 
$V:=V_{e'}\srjr\dts\srjr V_1$ and the resulting 
open immersions $V_{e'}=\lvP_{V_{e'}}\!\hra\!\dts\!
\hra\lvP_{V_1}\!\hra\lvP_{\JJ{e'}}$ and 
$\lvP_{\!f\!,V_{e'}}\!\hra\dts\hra\lvP_{\!f\!,V_1}\!\hra
\lvP_{\!f\!,\JJ{e'}}$, the following hold:
\vskip2pt
\itm{25}{
\item[1)$'$] $\,\tlclY_\mu\to V_\mu$ is a 
projective smooth integral $V_\mu$-variety of 
dimension $e'+1-\mu$ for $1\!\leqslant\!\mu\!\leqslant\! e'\hb1$.
\vskip2pt
\item[2)$'$] $\,\tlclY_\mu\to\clY_\mu:=\tlclY_{\mu-1}
\times_{\lvP_{\!\JJ{e'}}}\hb2\lvP_{V_\mu}$ is a 
projective surjective $V_\mu$-morphism for 
$1\!\leqslant\!\nu\!\leqslant\! e'\hb1$.
\vskip0pt
\item[3)$'$] Setting $\tlclY_{\mu,V}:=
\tlclY_\mu\times_{\lvP_{\JJ\mu}}\hb3V$ for 
$0\leqslant\mu\leqslant e'\!$, one gets canonically
a sequence of projective surjective flat morphisms 
of $V$-curves $\tlclY_{e'\!,V}\to\dts\to\tlclY_{1,V}\to
\tlclY_{0,V}\to\lvP_{\!f\!,V}\,$.
}
\vskip2pt
Further, $\bma':=(a_1,\dts,a_{e'})\in V(\ko)$,
$\bma'_\mu:=(a_i)_{i\in \JJ\mu}\in \Vu\mu(\ko)$, and
the fiber $\tlclY_{\bmam_\mu}$ of $\,\tlclY_\mu\to \Vu\mu$ 
at $\bma'_\mu\in\Vu\mu$ is a projective smooth $\ko$-variety. 
Moreover, since each $\tlclY_{\mu,\bmam'_\mu}$ is an 
irreducible component of the projective smooth
$\ko$-variety $\tlclX_{\nu,\bmam_\nu}$, the 
condition~$(\ddag)_{\bmam}$ above implies:
\vskip5pt
\centerline{$(\ddag)_{\bmam'}$ \ {\it For all 
$\,\clJ=(\JJ\mu)_\mu$, $\tlclY_0$ and all the 
$\ko$-varieties $\tlclY_{\bmam'_\mu}$ have 
projective smooth $\Uo$-models.\/}}
\vskip7pt
\noindent
In particular, since $e'<e$, we can apply the induction 
hypothesis for $\bm\beta=\partial_{x_0}(\tl{\bm\alpha})$
as follows: Recall that  the canonical
projective morphism $\tlclY_0:=\tlclX_{1,\bmam_1}\to 
\tlclX_{0,\bmam_1}=:Y_0$ is the fiber of
$\tlclX_1\to\tlclX_{0,U_1}$ at $\bma_1\in U_1(\ko)$,
and the corresponding extension $L_0=\kappa(x_0)
\hra\kappa(x_1)=\tlL_0$ has degree $f(x_1|x_0)$ 
prime to $n$. Hence the image 
$\tl{\bm\beta}:=\res(\bm\beta)$ of $\bm\beta$ 
over $\tlL_0$ is non-trivial. Let $C'$ be the 
generic fiber of $\tlclY_{0,V}\to V$, hence 
of $\tlclY_0\to\lvPtp$, and set $\bmu':=\bmt'-\bma'
=(t_i-a_i)_{1\leqslant i\leqslant e'}$. Then $C'$ 
is a projective regular $\ko(\bmt')$-curve, and the 
images $(f\!,\bmu')\mapsto(\oli f\!,\oli{\bmu}')
\mapsto(f'\hb4,\bmu')$ under $\clO_{x_0}\to\kappa(x_0)
\hra\tlL_0$ satisfy: $\tl{\bm\beta}=\uop\cpm \uo\cpm 
\cppr{u'_1}{u'_{e'}}\cpm f'\!$. Hence by the 
induction hypothesis, there is $P'\in C'$ such 
that $\partial_{P'}(\tl{\bm\beta})\neq0$ and 
$v_{P'}(f')\notin n\cdot v_{P'}(\tlL_0)$. 
In particular, $y:=P'\in C'\subset\tlclY_0$ is a point 
of codimension one. Recalling that 
$\tlclY_0=\tlclX_{1,\bmam_1}=\tlclX_{\bmam_1}$ 
is the fiber of $\tlclX_1$ at $\bma_1\in\ U_1(\ko)$, it 
follows that $y\in\tlclX_1$ is a point of codimension two.
By the discussion before Case~$e=1$ above, there exists 
$x'_1\in\tlclX_1(y)$ with  $\partial_{x_1}(\tl{\bm\alpha})\neq0$
and $\ko(t_e)\subset\clO_{x'_1}$. On the other hand, 
$\ko(\bmt')\subset\kappa(y)\subset\kappa(x_1)$,
hence finally, $\bmt=(\bmt'\!,t_e)$ consists of 
$w_{x_1}$-units. Conclude as in Case~1.
\end{proof}
%
%
%
\noindent
B) \ {\it A criterion for $\DDf$ to be non-empty\/}
\vskip4pt
\begin{notations/remarks} 
\label{nota3}
Recalling that~$\vo$ denotes the finite places 
of $\ko$, we supplement the previous 
Notatations/Remarks~\ref{nota1} as  follows. For 
every $\delta,\delta_1,\dots,\delta_e\in\ko^\times$ 
we define:
\vskip5pt
\itm{25}{
\item[1)] $U_\delta:=\{a\in\ko^\times\mid
\vo(\delta)\neq0\hbox{ or } \vo(n)>0 \Rightarrow
\vo(a-1)>2\cdot\vo(n)\ \forall\ \vo\}$, the 
\defi{$n,\!\delta$-units}.
\vskip3pt
\item[2)] $\Sod=\{a\in\ko^\times\mid \vo(\delta)\neq0
\Rightarrow \vo(a)\neq0\}$, the \defi{non-$\delta$-units}.
We notice the following:
\vskip2pt
\itm{15}{
\item[a)] $U_\delta$ is a subgroup of 
$\ko^\times\!$, and $\Sod$ is a $U_\delta$-set, 
i.e., $U_\delta\cdot\Sod=\Sod$. 
\vskip2pt
\item[b)] For every finite set $A\subset\ko$ 
there exist ``many''~$\delta\in\ko^\times$ such 
that $A\cap\Sod=\vid$.
}

\vskip4pt
\item[3)] $U^{^\bullet}_{\!\delta}:=
U_\delta\times\ko^\times\!$, and \ 
$\bm\Sigma_{\bm\delta}:=\times_i\,\Sigma_{\delta_i}$ for 
$\bm\delta:=(\delta_1,\dots,\delta_e)$.
\vskip2pt
\item[4)] For $\bma=(a_i)_i 
\in\bm\Sigma_{\bm\delta}$, set $\bmu:=\bmt-\bma$, 
and for $\delta_0\in\ko^\times\!$ 
and $\bmu^\nxx_0\!:=(u'_0,u_0)\in U^{^\bullet}_{\delta_0}$, 
define
\vskip5pt
\centerline{$\bm\alpha_{\bmum_0^\nxx,\bmam,f}\!:=
\uop\cpm\uo\cpm\cppr{u_1}{u_e}\cpm\!f 
\in \HH^{d+1}\big(K,\lvZ/n(d)\big)$}
\vskip5pt
\noindent
and consider the subgroup $\,\HH_{\delta_0,\bmam,f}:=
\big\langle\bm\alpha_{\bmum_0^\nxx,\bmam,f}
\mid\bmu^\nxx_0\in U^{^\bullet}_{\delta_0}\,\big\rangle
\subset\,\HH^{d+1}\big(K,\lvZ/n(d)\big)$.
}
\end{notations/remarks}
%
%
\begin{keylemma}
\label{keylmm}
In the above notations, the following are equivalent:
\vskip2pt
\itm{25}{
\item[{\rm i)}] $\DDf$ is non-empty.
\vskip2pt
\item[{\rm ii)}]  $\exists\,\delta_1\ha2\forall\ha1
a_1\!\in\!\Sigma_{\delta_1}\dts\exists\,\delta_e
\ha2\forall\,a_e\!\in\hb1\Sigma_{\delta_e}\forall\,
\delta_0\,\exists\,(\uop,\uo)\!\in\!U^{^\bullet}_{\delta_0}$ such that 
$\,\bm\alpha_{\bmum^{\phantom.}_0,\bmam,f}\neq0$.
}
\end{keylemma}
\begin{proof} The proof follows easily from
Proposition~\ref{prop2} along the following 
lines:
\vskip2pt
To i) $\Rightarrow$ ii): Given i), by 
Proposition~\ref{prop2}, $\exists\,\Ut\subset\lvAt$ 
Zariski open dense s.t.\ $\,\HH_\Uaf\neq0$ for
all $\bma\in \Ut(\ko)$, $\Uo\subset\So$.
Set $\bmt_i=(t_1,\dts,t_i)$. Then  
$\varpi_i:\lvAt\to\lvAti$~def\-ined by 
$\ko[\bmt_i]\hra\ko[\bmt]$ are Zariski open maps. Hence 
$\Sigma_i:=\varpi_i\big(\Ut(\ko)\big)\subset\ko^i$
is Zariski open dense,~and~one~has:
\vskip2pt
\itm{25}{
\item[a)] If $\bma_i\in\Sigma_i$, and 
${\Ut}_{\hb1,\bmam_i}\subset{\lvAt}_{\hb1,\bmam_i}
\hra\lvAt$ are the fibers of $\Ut\subset\lvAt$ at 
$\bma_i$ under $\varpi_i:\lvAt\to\lvAti$, then 
at the level of $\ko$-rational points, one has 
canonical identifications
\vskip5pt
\centerline{${\Ut}_{\hb1,\bmam_i}(\ko)\subset
{\lvAt}_{\hb1,\bmam_i}(\ko)=\varpi_i^{-1}
(\bma_i)=\bma_i\times\ko^{(e-i)}\subset\ko^e=\lvAt(\ko)$.}
\vskip5pt
\item[b)] In particular, ${\Ut}_{\hb1,\bmam_i}(\ko)\subset
\bma_i\times\ko^{(e-i)}$ is a Zariski open dense subset for
all $\bma_i\in U_i$.
}
\vskip3pt
\noindent
Proceed by induction on $i=1,\dots,e$ as follows:
\vskip3pt
\underbar{Step 1}. $i=1$: Since 
$\Sigma_1:=\varpi_1\big(\Ut(\ko)\big)\subset\ko$ 
is Zariski open dense, $A_1:=\ko\backslash\Sigma_1$ 
is finite. Hence $\exists$ $\delta_1\in\ko^\times$
such that $\Sigma_{0,\delta_1}\cap A_1=\vid$,  
thus $\Sigma_{0,\delta_1}\subset\Sigma_1$. 
Then $\forall$ $a_1\in\bm\Sigma_{\delta_1}$,
set $\bma_1:=(a_1)$.
\vskip3pt
\underbar{Step 2}. $i\,\Rightarrow\,i+1$: Suppose that 
$\bma_i=(a_1,\dots, a_i)\in\varpi_i\big(\Ut(\ko)\big)$ 
is inductively constructed. Let 
$\varpi_{i+1,i}:\lvAtiu\to\lvAti$ be the canonical 
projection. Then $\varpi_i=\varpi_{i+1,i}\circ\varpi_{i+1}$, 
and all the projections involved are open surjective. 
Further, by the discussion above, one has that
$U_{\bmam_i}:={\Ut}_{\hb1,\bmam_i}(\ko)\subset
\bma_i\times\ko^{(e-i)}$ is Zariski open dense,
and therefore $\varpi_{i+1,i}\big(U_{\bmam_i})\subset
\bma_i\times\ko$ is a dense open subset. Hence
there exists $\Sigma_{i+1}\subset\ko$ cofinite
such that $\bma_i\times\Sigma_{i+1}\subset
\varpi_{i+1,i}\big(U_{\bmam_i})$, thus  
$\bma_i\times\Sigma_{i+1}\subset U_{\bmam_i}
\subset\bma_i\times\ko$ is a Zariski open dense subset. 
In particular, $A_{i+1}:=\ko\backslash\Sigma_{i+1}$ is finite. 
Hence $\exists$ $\delta_{i+1}\in\ko^\times$ such that 
$\Sigma_{\delta_{i+1}}\cap A_{i+1}=\vid$, and in
particular, $\Sigma_{\delta_{i+1}}\subset\Sigma_{i+1}$.  
Then $\forall$ $a_{i+1}\in\Sigma_{\delta_{i+1}}$,
setting $\bma_{i+1}:=(\bma_i,a_{i+1})$, one has:
$\bma_{i+1}\in\varpi_{i+1}
\big({\Ut}_{\hb1,\bmam_i}(\ko)\big)\subset
\varpi_{i+1}\big(\Ut(\ko)\big)$.
This completes the proof of the induction step,
thus of the implication i) $\Rightarrow$ ii).
\vskip5pt
To ii) $\Rightarrow$ i): \ 
Let $\Ut\subset\lvAt$ be a Zariski dense open
subset. Then condition ii) of the Key 
Lemma~\ref{keylmm} above, implies that there
is $\bma\in\Ut(\ko)$ such that
$\forall\, \Uo\subset\So$ one has $\HH_\Uaf\neq0$.
Hence condition~iii) from Proposition~\ref{prop2} 
is satisfied, concluding that $\DDf\neq\vid$.
\end{proof}
\section{Uniform definability of the geometric 
            prime divisors of $K$} 
\noindent
In this section we work in the context and notation
of the previous sections, but specialize to the
case $n=2\neq\chr$. In particular, for 
$\bma:=(a_1,\dts,a_r)$ with 
$a_i\in K^\times\!$, by the Milnor Conjecture, 
$a_1\cpm\dts\cpm a_r\in \HH^r\big(K,\lvZ/n(r-1)\big)$
is trivial iff $0\in q_{\bmam}(K)$. Therefore:
\vskip5pt
\centerline{\it $a_1\cpm\dts\cpm a_r=0$ 
is first-order expressible by
$\exists$ $(x_1,\dts,x_{2^r})\neq0$ s.t.\ 
$q_{\bmam}(x_1,\dts,x_{2^r})=0$.\/}
\vskip5pt
Let $K$ satisfy Hypothesis~(H),
$\ko\subset K$ be a (relatively algebraically 
closed) global subfield, $e=\td(K|\ko)-1>0$, 
$(f\!, t_1,\dots,t_e)$ be algebraically independent 
functions over $\ko$, such that each $t_i$ is an
$n^{\rm th}$ power in $K$. Then $K$ is the 
function field of a projective normal curve 
$C$ over $\ko(\bmt)$, and $f$ is a non-constant 
function on $C$. Finally, recalling the context 
of the Key Lemma~\ref{keylmm}, 
let $\bmx:=(x_1,\dts,x_{2^{d+1}})\neq(0,\dts,0)$ 
be a system of $2^{d+1}$ variables, and  consider 
the following uniform first-order formula:
\[
\varphi(\fbmt)\,\equiv\,
\scalebox{1.0}[1.0]{$\exists\,\delta_1\ha2
\forall\ha1a_1\!\in\!\Sigma_{\delta_1}\!\dots\exists\, 
\delta_e\ha2\forall\,a_e\!\in\hb1\Sigma_{\delta_e}\forall\,
\delta_0\,\exists\,(\uop,\uo)\!\in\!U^{^\bullet}_{\delta_0}
\,\forall\,\bmx\hb1:\, q_{\bmum^{\phantom.}_0,\bmum,f}(\bmx)\neq0$.}
\]
%
%
%
\begin{keylemmarev} 
\label{keylmmrev}
In the above notation, the following are equivalent:
\itm{25}{
\item[{\rm i)}] $\DDf$ is non-empty.
\vskip2pt
\item[{\rm ii)}] $\varphi(\fbmt)$ holds in $\tlK=K[\mu_4]$.
}
\end{keylemmarev}
\begin{proof}
As explained above, this is just a reformulation 
of Key Lemma~\ref{keylmm}. 
\end{proof}
\noindent
Our final aim in this section is to show that the prime 
divisors of $K|\ko$ are uniformly first-order definable. 
Precisely, recalling the $(d+1)$-fold Pfister forms 
$q_{\bmum_0,\bmum,f}$ defined using $f\!,\bmt$ 
as introduced above, the Recipe~\ref{therecipe} 
below gives a proof of the following: 
\begin{theorem}
\label{mainthm}
There exist explicit formulae $\val_d$ which
uniformly define the geometric prime divisors of 
finitely generated fields $K$ with $\chr(K)\neq2$ 
and $\,\dim(K)>1$ as follows
\[
\val_d\big(\,\xx\ha1; f\!,\ha1\bmt,\bm\xi, 
\bmdl, (a_i)_i\in\bmSgodudlm, \delta_0, 
\bmu^\nxx_0\in U^{^\bullet}_{\!\delta_0},
q_{\bmum^\nix_0,\bmum,f}(\bm x_{\bm\xi})\big).
\]
\end{theorem}
\noindent
The proof of Theorem~\ref{mainthm} follows 
from the Recipe~\ref{therecipe} below.
\vskip9pt
%
%
\noindent
A) \ {\it The uniformly definable subsets $\,\Tft$, $\Tfto$ 
and semi-local subrings $\eua_\ft \subset R_\ft$ of $\,K$} 
\vskip5pt
\noindent
\begin{notations/remarks}
\label{nota4}
We supplement Notations/Remarks \ref{nota1},
\ref{nota2}, 
as follows. 
\vskip2pt
\itm{25}{
\item[1)] For $\veps\in K$, set $K_\veps:=K[\hb1\root n\of\veps\,]$
and consider $\res_\veps:\HH^{d+1}\big(K,\lvZ/n(d)\big)\to
                     \HH^{d+1}\big(\Kvx,\lvZ/n(d)\big)$.
\vskip2pt
\item[2)] Let $C_\veps\to C$, $P_\veps\mapsto P$, 
be the normalization of $C$ in the field extension 
$\Kvx\,|\,K$, and denote by $D\!_{f\hb1,\ha1\veps}$ 
the set of all $P_\veps\in C_\veps$ such that 
$v_{P_\veps}(f)\not\in n\ha1\cdot\ha1v_{P_\veps}\Kvx$. 
\vskip2pt
\item[3)] For every $P\in C$, let $U_P:=\clO_P^\times$
be the $P$-units, and let $U\!_P\subset\clO_{v_P}$ 
be the $v_P$-units. Then for $\veps\in K^\times$ one has:
(i) $\veps\in U\!_P\!\cdot\!K^{\lower2pt\hbox{\Large$\cdot$}n}$
\ iff \  (ii) $v_P(\veps)\in n\hb1\cdot\hb1 v_P K$.
} 
\end{notations/remarks}
\begin{lemma}
\label{keylmm1}
In the above notations, 
$\veps\in\cup_{P\in\DDf}\,U\!_P\!
\cdot\!K^{\lower2pt\hbox{\Large$\cdot$}n}\,$  
$\Leftrightarrow$ $\,\varphi(\fbmt)$ holds in $K_\veps$.
\vskip2pt
\noindent
In particular, 
$\,\Tft:=\cup_{P\in\DDf}\,U\!_P\!\cdot\!
K^{\lower2pt\hbox{\Large$\cdot$}n}\subset K$ 
are uniformly first-oder definable in $K$ as follows:
\vskip3pt
\centerline{$\Tft=\{\veps\in K\mid\varphi(\fbmt) 
\hbox{ holds in } \tlK_\veps=K_\veps[\mu_4]\,\}$}
\end{lemma} 
\vskip1pt
\begin{proof} 
To $\Rightarrow\,$: Let $\veps\in\cup_{P\in\DDf}\,
U\!_P\cdot K^{\lower2pt\hbox{\Large$\cdot$}n}$ 
be given, and $P\in\DDf$ be such that $\veps\in U\!_P
\cdot K^{\lower2pt\hbox{\Large$\cdot$}n}$. Then 
$P$ is unramified in the extension $K_\veps|K$.
Hence if $C_\veps\to C$ is the normalization of $C$ 
in the field extension $K_\veps\hla K$, it follow that
$v_{P_\veps}(f)=v_P(f)$ is prime to $n$. Hence
$D\!_{f\hb1,\ha1\veps}\neq\vid$, and therefore,
by Key Lemma~\ref{keylmmrev} follows that 
condition~ii) is satisfied over $K_\veps$. 
Let $\koveps=\oli\ko\cap K_\veps$ be the field 
of constants of $K_\veps$. Then $\koveps|\ko$ is 
a finite field extension, and therefore, for every
$\delta_\veps\in\koveps^\times$ there is $\delta\in\ko^\times$
such that for all $\voeps$ and $\vo:=\voeps|_{\ko}$ one
has: $\voeps(\delta_\veps)\neq0$ iff $\vo(\delta)\neq0$.
In particular, $\Sigma_{\delta_\veps}\cap\ko=\Sigma_\delta$.
Therefore, condition~ii) for $K_\veps$ implies condition~ii)
for $K$. 
\vskip3pt
To $\Leftarrow\,$: Let
$\veps\not\in\cup_{P\in\DDf}\,U\!_P\cdot 
K^{\lower2pt\hbox{\Large$\cdot$}n}$ that is,
$\veps\not\in U\!_P\cdot K^{\lower2pt
\hbox{\Large$\cdot$}n}$ for all $v\in\DDf$. Then
by~Notations/Remarks~\ref{nota3},~3), one has
$v_P(\veps)\not\in n\cdot v_P(K)$ for all $P\in\DDf$. 
Hence for all $P\in\DDf$, and any prolongation
$P_\veps$ to $K_\veps$ one has $e(P_\veps|P)=n$,
thus $v_{P_\veps}(f)=e(P_\veps|P)\,v_P(f)
\in n\cdot v_{P_\veps}(K_\veps)$. Further, since 
$v_P(f)\in n\cdot v_P(K)$ for $P\not\in\DDf$, one 
has $v_{P_\veps}(f)\in n\cdot v_{P_\veps}(K_\veps)$ 
for $P_\veps\mapsto P\not\in\DDf$, thus conclusing 
that $v_{P_\veps}(f)\in n\cdot 
v_{P_\veps}(K_\veps)$ for all $P_\veps\in C_\veps$. 
On the other hand, by hypothesis~ii), applying 
Key Lemma~\ref{keylmmrev} to $K_\veps$, it follows that 
$\DD_{f\hb1,\ha1\veps}$ is non-empty, contradiction!
\end{proof}
\begin{notations/remarks}
\label{nota5}
In the notations from Lemma~\ref{keylmm1} above,
we have the following:
\vskip2pt 
\itm{25}{
\item[1)] Let $\eta\in K\ha2\backslash\,\Tft$
be given. Then by Notations/Remarks~\ref{nota4},~3), 
it follows that for all $P\in\DDf$ one has:
$v_P(\eta)\not\in n\!\cdot\!v_{P}(K)$.
In particular, $v_P(\eta)\neq0$, and therefore one has:
\vskip2pt
\itm{20}{
\item[-] If $v_P(\eta)>0$, then 
$\eta-1\,\in\, \eum_P-1\subset U_P\subset\Tft$,
hence finally $\eta-1\,\in\,\Tft$. 
\vskip2pt
\item[-] If $v_P(\eta)<0$, then 
$v_P(\eta-1)=v_P(\eta)\not\in n\!\cdot\!v_P(K)$. 
Therefore, by the discussion at Notations/Remarks 
\ref{nota4},~3), it follows that 
$\eta-1\not\in U_P\cdot K^{\lower2pt\hbox{\Large$\cdot$}n}$.
}
\vskip2pt
\item[2)] Conclude that for $\eta\in K$ the
conditions (i), (ii) below are equivalent: 
\vskip5pt
\centerline{\ha{20} (i) $\eta, \eta-1\not\in\,\Tft\,$;
\ (ii) {\it $v_P(\eta)<0$ 
and 
$v_P(\eta)\not\in n\cdot v_P(K)$ for all $P\in\DDf$.\/}}
\vskip5pt
\item[3)] Hence $\Tfto\hb2:=\hb1\big\{\hb1\xi\in K\,|\,
{\textstyle{1\over\xi}, {1\over\xi}\hb1-\hb{2}1}\!
\not\in\Tft\big\}$ are \hbox{uniformly 
definable,~and}
\[
\hbox{\it $\xi\in\Tfto$ \ iff \ 
$\forall\ha3P\in\DDf\,$ one has: $\,v_P(\xi)>0, \,
v_P(\xi)\not\in n\cdot v_P(K)$.\/}
\leqno{\ha{20}(*)}
\]
}
\end{notations/remarks}
\begin{lemma}
\label{keylmm3}
In the above notation, one has
$\eua_\ft\!:=\cap_{P\in\DDf}\,\eum_P=\Tfto-\Tfto$.
Hence $\eua_\ft\subset K$ is uniformly definable,
thus so is the subring 
$R_\ft=\cap_{P\in\DDf}\clO_P=\{\,r\in K\ha2|\ha2 
     r\ha1\cdot\ha1\eua_\ft\subset\eua_\ft\}$~of~$K\!$.
\end{lemma}
\begin{proof} We first prove the equality
$\cap_{P\in\DDf}\,\eum_P=\Tfto-\Tfto$. 
For the inclusion ``$\,\subset\,$'' notice that
$\Tfto\subset\eum_P$, $P\in\DDf$ by 
Notations/Remarks~\ref{nota4},~3) above. Hence 
$\Tfto-\Tfto\subset\eum_P-\eum_P=\eum_P$, 
$P\in\DDf$, thus finally one has 
$\Tfto-\Tfto\subset\eua_\ft$. For the converse 
inclusion ``$\,\supset$'' let $\xi\in\eua_\ft$ be arbitrary. 
Since $\eua_\ft=\cap_{P\in\DDf}\,\eum_P$, 
it follows by Notations/Remarks~\ref{nota4},~3), 
above that $v_P(\xi)>0$ for all $P\in\DDf$. 
Hence by the weak approximation lemma, it 
follows that there exists $\xi'\in K$ such that 
both $\xi'$ and $\xi''\hb3:=\xi'\hb3-\xi$ satisfy 
$v_P(\xi'), v_P(\xi'')=1$. In particular, by 
Notations/Remarks~\ref{nota4},~3), one has
$\xi'\!,\xi''\in\Tfto$, hence $\xi=\xi'\hb3-\xi''\in\Tfto-\Tfto$. 
Concerning the assertions about $R_\ft$, the first
row equalities are well known basic valuation 
theoretical facts (which follow, e.g.\ using the weak 
approximation lemma). 
\end{proof}
\begin{corollary}
\label{corr}
In the above notation, let $\DDf^0\!:=\{P\in\DDf\,|\,
v_P(f)>0\}$ and $f'\hb3:=f/(f+1)$. Then
$R^0_\ft\!:=\cap_{P\in\DDf^0}\,\clO_P\!=\!R_\ft
\cdot R_{f'\hb3,\ha1\bmtm}\hb3:=\!\{rr' |\,r\!\in\! R_\ft,
r'\hb3\in\! R_{f'\hb3,\ha1\bmtm}\}$ is uniformly 
first-order definable.
\end{corollary}
\begin{proof} Noticing that $f\!,f'$ have the same
zeros and no common poles, the assertion of the 
Corollary~\ref{corr} follows from Lemma~\ref{keylmm3}
by the weak approximation lemma. 
\end{proof}
%
%
\vskip5pt
\noindent
B) {\it Defining the $k$-valuation rings of $K|k$\/}
\vskip5pt
In the notations and hypotheses the previous sections,
recall that $K=k(C)$ for some projective normal 
$k$-curve $C$. By Riemann-Roch, for every closed 
point $P\in C$ there are integers $m\gg0$ prime to $n=2$ 
and functions $h\in K\tms$ such that $(h)_\infty=m\,P$.
Hence setting $f=1/h$, by Corollary~\ref{corr} one has:
\vskip5pt
\centerline{$\clO_P=R^0_{\ft}=R_\ft\cdot R_{f'\hb3,\ha1\bmtm}$,
\quad
$\eum_P=\{r\in K\mid r\in\clO_P,
        \,r^{-1}\not\in\clO_P\}$.}
\vskip5pt\noindent
Hence we have the following {\it uniform first-oder 
recipe\/} to define the prime $\ko$-divisors of $K|k$:
\begin{recipe} 
\label{therecipe}
Recall $\varphi_d$, $\abs\psi(\xx)$, 
$\psi_r(\tx_1,\dts,\tx_r)$ from the Introduction. 
If $\dim(K)=1$, then the prime divisors of $K$ 
are uniformly first-order definable by the formulae 
$\val_1$ given by \nmnm{Rumely}~[Ru]; and if 
$\dim(K)=2$, the {\it geometric\/} prime divisors 
of $K|\ko$ are uniformly first-order definable by 
the formulae $\val_2$ given by \nmnm{Pop}~[P4].
Next let $K$ satisfying Hypothesis~(H) from
Introduction, $\ko\subset K$ be (maximal)
global subfields, and $e:=\dim(K)-2=\td(K|\ko)-1>0$. 
We construct  
$
\val_d\big(\,\xx\ha1; f\!,\ha1\bmt, \bm\xi, 
\bmdl, (a_i)_i\in\bmSgodudlm, \delta_0, 
\bmu^\nxx_0\in U^{^\bullet}_{\!\delta_0},
q_{\bmum^\nix_0,\bmum,f}(\bm x_{\bm\xi})\big)
$  
concretely along the steps: 
\vskip4pt
\itm{25}{
\item[1)] Consider the systems $\bmt:=(t_1,\dts,t_e)$,
of $\ko$-algebraically independent elements 
of $K$~with $t_i$ squares in $K$. These are
uniformly definable using the formula 
$\psi_e(\tx_1,\dts,\tx_e)$ over $\ko$. 
\vskip2pt  
\item[2)] $\clP_{\ko,\defi{val}}:
=\{(\bmt, f)\!\in\! K^{e+1}\,|\,R^0_\ft\subsetneq K \ 
\scalebox{.95}[1]{{is a valuation ring}}\}$
is {\it uniformly first-order definable.\/} 
\vskip2pt
\item[3)] Finally, the above $R^0_\ft$ are 
valuation rings of prime $\ko$-divisors of 
$K|\ko$, and conversely, for every prime 
$\ko$-divisor $w$ of $K|\ko$ there are pairs 
$(\fbmt)\in\clP_{\ko,\defi{val}}$ such that 
$\clO_w=R^0_\ft$.
}
\vskip2pt
Conclude that the geometric prime divisors of 
$K$ are uniformly first-order definable via
$\clP_{\ko,\defi{val}}$.
\end{recipe}
\section{Proof of the Main Theorem}
\noindent
A) \ {\it First proof\/}: Using \nmnm{Scanlon}~[Sc]
\vskip5pt
A first proof follows simply from~\nmnm{Scanlon},
Theorem~4.1 and Theorem~5.1, applied to the case 
of characteristic $\neq2$, using Theorem~\ref{TheMThm}
for the definability of valuations (which is essential
in both Theorem~4.1 and Theorem~5.1 of loc.cit.). 
This proof also shows that finitely generated 
fields of characteristic $\neq2$ are \textit{\textbf{bi-interpretable 
with the arithmetic.}\/}
\vskip7pt
\noindent
B) {\it Second proof\/}: Using \nmnm{Aschenbrenner%
--Kh\'elif--Naziazeno--Scanlon}~[AKNS]
\vskip5pt
Recall that one of the main results of 
[AKNS] asserts that the {\it finitely 
generated infinite domains $R$ are bi-interpretable 
with arithmetic,\/} see Theorem in the Introduction
of loc.cit. In particular, the isomorphism type of any 
such domain is encoded by a sentence $\istp R$. 
Thus the Main Theorem from the Introduction
follows from the following stronger assertion:
\begin{theorem}
Let $\kappa_0$ a prime field, $\chr(\kappa_0)\neq2$,  
$R_{\kappa_0}\subset\kappa_0$ be its prime subring, 
and $\clT=(t_1,\dots,t_r)$ be independent variables. 
Then the integral closures $R\subset K$ of $\,R_{\kappa_0}[\clT]\,$ 
in finite field extensions $\,K\,|\,\kappa_0(\clT)$ are 
uniformly first-order definable finitely generated domains.
\end{theorem}
\begin{proof} First, by the Finiteness Lemma,
$R$ is a finite $R_0[\clT]$-module, hence finitely 
generated as ring. The uniform definability of $R$ 
is though more involved, and uses the uniform 
definability of \defi{generalized geometric prime 
divisors} of $K$ combined with \nmnm{Rumely}~[Ru]. 
\begin{lemma} 
\label{lmmuno}
Let $A$ be an integrally closed domain, and
$\clV$ be a set of valuations of the fraction
field $K_A:=\Quot(A)$ 
such that $A=\cap_{v\in\clV_A}\,\clO_v$. 
Let $B$ be the integral closure of $A$ in an 
algebraic extension $K_B|K_A$, and $\clW$ 
be the prolongation of $\clV$ to $K_B$. 
Then $B=\cap_{w\in\clW}\,\clO_w$.
\end{lemma}
\begin{proof} Klar, left to the reader. \end{proof}
\noindent
Let $R_0$ be an integrally closed domain, 
$L_0:=\Quot(R_0)$, and $\clV_0$ be a 
set of valuations of $L_0$ such that 
$R_0=\cap_{v\in\clV}\clO_v$. Let $L_1|L_0(t)$
be a finite field extension, $R_1\subset \tl R_1\subset L_1$
be the integral closures of $R_0[t]\subset L_0[t]$ 
in $L_1$. Then $\kappa(P)$ are finite field extensions 
of $L_0$, $P\in\Max(\tl R_1)$ and let $\clV^P$ be 
the prolongation of $\clV_0$ to $\kappa(P)$. 
Finally let $\clV_1$ be the set of all
the valuations of the form $v_1:=v^P\!\circ v_P$ 
with $v_P$ the valuation of $P\in\Max(\tl R_1)$,
and $v^P\in\clV^P\!$. Then $v^P=v_1/v_P$, 
$v_1 L_1=v^P\!L^P\times\lvZ$ lexicographically 
ordered, and $L_1v_1=L^P \!v^P\!$. Further, 
the canonical restriction map 
${\rm Val}(L_1)\to{\rm Val}(L_0)$ gives 
rise to a well defined~surjective~maps:
\[
\clV_1\to\clV^P\to\clV,\quad v_1\mapsto 
   v^P\mapsto v_0:=v^P|_{L_0}=v_1|_{L_0}.
\]
\begin{lemma}
\label{lmmdue}
In the above notation, one has 
$R_1=\cap_{v_1\in\clV_1}\,\clO_{v_1}$.
\end{lemma}
\begin{proof} Lemma~\ref{lmmuno}
reduces the problem to the case $L_1=L_0(t)$,
$R_1=R_0[t]$. For $v_1=v^P\!\circ v_P$, 
$\clO_{v_1}\subset\clO_{v_P}$, hence
$\cap_{v_1\in\clV_1}\,\clO_{v_1}\subset 
\cap_{P}\,\clO_{v_P}=L_0[t]$. Thus it is
left to prove that $f(t)\in L_0[t]$ satisfies:
$v_1(f)\geq0$ for all $v_1\in\clV_1$ \ iff \
$f\in R_0[t]$. This easy exercise is left to
the reader.
\end{proof}
\begin{lemma}
\label{lmmtre}
Suppose that all the valuation
rings $\clO_P$, $P\in\Max(\tl R_1)$
and $\clO_{v^P}$, $v^P\in\clV^P$ are 
$(\hb1$uniformly$\ha1)$ first-oder definable. Then 
so are $\clO_{v_1}$, $v_1\in\clV_1$ and 
$R_1=\cap_{v_1\in\clV_1}\,\clO_{v_1}$.
\end{lemma}
\begin{proof}
For $v_1=v^P\!\circ v_P$, 
one has $\clO_{v_1}=\pi_P^{-1}(\clO_{v^P})$,
where $\pi_P:\clO_P\to\kappa(P)=:L^P\!$, etc. 
\end{proof}
\noindent
Finally, all of the above can be performed
inductively for systems of variables 
$\clT:=(t_1,\dts,t_r)$, $L_r$ finite field 
extension of $L_0(\clT)$, $R_r\subset\tl R_r\subset L_r$ 
the integral closures of $R_{r-1}[t_r]\subset L_{r-1}[t_r]$, 
thus leading to the corresponding sets of all 
valuations $\clV_r$ of $L_r$ the form 
$v_r=v_{r-1}^P\circ v_P$, where 
$P\in\Max(\tl R_r)$ and $v_{r-1}^P$ lies in the 
prolongation $\clV_{r-1}^P$ of $\clV_{r-1}$ to 
$\kappa(P)$. 
\begin{lemma}
\label{lmmquad}
In the above notation, one has 
$R_r=\cap_{v_r\in\clV_r}\,\clO_{v_r}$.
Further, if all the valuation rings $\clO_P$, 
$P\in\Max(\tl R_r)$ and $\clO_{v^P}$, 
$v_P\in\clV_{r-1}^P$ are $(\hb1$uniformly\ha1$)$ 
first-order definable, then so are the valuation
rings $\clO_{v_r}$, $v_r\in\clV_r$ and 
$R_r=\cap_{v_r\in\clV_r}\,\clO_{v_r}$.
\end{lemma}
\begin{proof} 
Induction on $r$ reduces everything to $r=1$. 
Conclude by using Lemmas~\ref{lmmdue},~\ref{lmmtre}. 
\end{proof}
Coming back to the proof of Theorem~5.1, 
one has $R_{\kappa_0}=\lvF_p$, $p>2$ or 
$R_{\kappa_0}=\lvZ$, and $R\subset K$ is the 
integral closure of $R_{\kappa_0}[\clT]$ in $K$. 
Hence Theorem~5.1 follows from 
Lemma~\ref{lmmquad}.
\end{proof}
\vskip5pt
\noindent
C) {\it Third proof\/}: A direct proof involving \nmnm{Rumely} [Ru]
\vskip5pt
We begin by mentioning that \nmnm{Pop}~[P4], 
Theorem~1.2 holds in the following more general 
form (which might be well known to experts, but 
we cannot give a precise reference). Namely, let 
$\clK$ be a class of function fields of projective 
normal geometrically integral curves $K=k(C)$ 
such that $k\subset K$ and the $k$-valuations rings 
$\clO,\eum$ of $K|k$ are (uniformly) first-order 
definable. Then for every non-zero $t\in K$, 
$e>0$, the sets
\[
\Sigma_{t,e}:=\{\clO,\eum\mid t\in\eum^e\!,\,
    t\not\in\eum^{e+1}\}
\]
are (uniformly) first-order definable subsets of the
set of all the valuation rings $\clO,\eum$. Hence 
given $N>0$, a function $t\in K^\times$ has 
$\deg(t):=[K:k(t)]=N$ iff the following hold: 
\vskip2pt
\itm{20}{
\item[i)] $\Sigma_{t,N+1}=\vid$ and 
$|\Sigma_{t,e}|\leq N$ for all $0<e\leqslant N$.
\vskip2pt
\item[ii)] $\dim_k\,\clO/\eum\leq N$ for all 
$\clO,\eum\in\Sigma_{t,e}$, and moreover:
$N={\textstyle\sum}_{0<e\leqslant N}
{\textstyle\sum}_{\clO,\eum\in
\Sigma_{t,e}}\,e\dim_k\clO/\eum$.
}
Hence there exists a (uniform) first-order 
formula $\deg_N(\eut)$ such that for $t\in K$ 
one has:
\vskip2pt
\itm{20}{
\item[-] ${\rm deg}_N(t)$ is true in $K$ iff 
$\,t$ has degree $N$ as a function of $K|k$, i.e.,
$[K:k(t)]=N$.
}
\vskip5pt
Now let $K$ satisfy Hypothesis~(H). The
isomorphism type of  $K$ is given by the data:
{\it 
\vskip2pt
\itm{30}{
\item[{\rm a)}] $\exists\,\ko\subset K$ with
$\ko=\oli k_{\mic0}\cap K$ and $\dim(\ko)=1$.
\vskip2pt
\item[{\rm b)}] $\exists\,(\bmt_e, t)$ a 
transcendence basis of $K|\ko$ with 
$k:=\ko(\bmt)$ satisfying $k=\oli k\cap K$.
\vskip2pt
\item[{\rm c)}] $\exists\,\,\fK\in\ko[\bm T_e, T\!, U]$ 
irreducible, $U$-monic $U$-separable, with coefficients
$\Sigma_\fK:=\{a_{\bmim}\}_{\bmim}$.
\vskip2pt
\item[{\rm d)}] $\exists\,u\in K\backslash k$ such 
that $[K:k(u)]=\deg_U(\,\fK)=:N_{\fK}$, i.e.,
$\deg_{\,\fK}(u)$ holds in $K|k$. 
}
}
\noindent
Consider pairs $\ko,\Sigma$ of global fields
endowed with finite systems of elements 
$\Sigma$. Then by \nmnm{Rumely}~[Ru], 
there exists a sentence $\istp{\ko,\Sigma}$ 
such that for all pairs $\ko',\Sigma'$ one 
has: $\istp{\ko,\Sigma}$ holds in $\ko',\Sigma'$ 
if and only if there exists a field isomorphisms 
$\imath:\ko\to\ko'$ with $\imath(\Sigma)=\Sigma'\!$. 
Further, inductively on $e$, let $\istp{\ko,\bmtm_e}$
be the sentence asserting that if $\bmt_e=(t_1,\dts,t_e)$ 
are $\ko$-algebraically independent in $K|\ko$, 
then $k:=\ko(\bmt_e)$ satisfies $k=\oli k\cap K$. 
Finally, consider the sentence $\istp K$
defined as follows:
\vskip2pt
\centerline{$\istp K \ \equiv$ \ a) $\wedge$ b)
$\wedge$ c) $\wedge$  d) $\wedge \ 
\istp{\ko,\Sigma_\fK} \wedge
\istp{\ko,\bmtm_e} \wedge \deg_{N_\fK}\!(u)$.}
\vskip2pt
Let $L$ be a finitely generated field such that
$\istp K$ holds in $L$. Then the assertions
a), b), c), d) together with $\istp{\ko,\Sigma_\fK}$
imply that there exists a subfield 
$\lo\subset L$ with $\lo=\oli l_{\mic0}\cap L$
and $\dim(\lo)=1$, and an isomorphism 
$\imath:\ko\to\lo$ such that setting 
$g_L:=g_L(\bm T_e, T, U):=\imath\big(\,\fK(\bm T_e, T, U)\big)$,
one has: $g_L(\bm T_e, T, U)$ is irreducible over $\lo$, and
there exist: First, a transcendence basis $(\bmt'_e,t')$ 
with $\bmt'_e:=(t'_1,\dts,t'_e)$ of $L|\lo$, 
such that $\istp{\lo,\bmtm'_e}$ holds in $L$, 
hence $l:=\lo(\bmt'_e)$ satisfies $l=\oli{\ha1l\ha1}\cap L$. Second, 
$u'\in L$ such that both $g_L(\bmt'_e, t'\!,u')=0$ and  
$\deg_{\,\fK}(u')$ hold in $L|l$, hence 
$[L:l(u')]=N_\fK$. Therefore, $\imath:\ko\to\lo$ 
together with $(\bmt_e,t,u)\mapsto(\bmt'_e,t'\!,u')$ 
give rise to a field embedding $\imath_K:K\to L$ 
such that $l=\lo(\bmt'_e)$ is relatively algebraically 
closed in $L$, and setting $L':=\imath_K(K)$, one 
has: $L'|l\hra L|l$ is a finite extension of 
function fields of curves over $l$ and
\[
[L:l(u')]=N_{f_K}=\deg_U(\,\fK)=\deg_U(g_L)=[L':l(u')].
\]
Hence we conclude that $L=L'$, thus $L=\lo(\bmt'_e,t'\!,u')$.
Therefore, the canonical embedding $\imath _K:K\hra L$ 
is actually a field isomorphism prolonging $\imath:\ko\to\lo$. 
\vskip10pt
\centerline{\NMNM{References}}
\vskip0pt


{\footnotesize
\begin{itemize}


\item[\lit{AKNS}] Aschenbrenner, M., Kh\'elif, A., 
          Naziazeno, E.\ and Scanlon, Th., {\it The logical complexity 
          of finitely generated commutative rings,\/} Int.\ Math.\ Res.\ 
          Notices (to appear). 

\item[\lit{Di}] Dittmann, Ph., {\it Defining Subrings in Finitely 
         Generated Fields of Characteristic Not Two,\/} 
         \vskip0pt\noindent
         See:  {\tt arXiv:1810.09333} [math.LO], Oct\ha{2}22, 2018.
         

\item[\lit{Du}] Duret, J.-L., {\it \'Equivalence \'el\'ementaire 
          et isomorphisme des corps de courbe sur un corps 
          algebriquement clos,\/} J. Symbolic Logic {\bf57} 
          (1992), 808--923.
      
\item[\lit{Ei}] Eisentr\"ager, K., {\it Integrality at a prime for 
           global fields and the perfect closure of global fields of 
           characteristic $p>2$,\/} J.\ Number Theory {\bf114} 
           (2005), 170--181.

\item[\lit{E-S}] Eisentr\"ager, K.\ and Shlapentokh, A., {\it Hilbert's 
           Tenth Problem over function fields of positive characteristic 
           not containing the algebraic closure of a finite field,\/} JEMS 
           {\bf19} (2017), 2103--2138.
           
\item[\lit{Hi}] Hironaka, H., {\it Resolution of singularities of an
          algebraic variety over a field of characteristic zero,\/}
          Annals of Math.\ {\bf 79} (1964), 109--203; 205--326.           

\item[\lit{Ill}] $\hbox{Illusie, L., {\it Compl\`exe de de Rham--Witt et 
              cohomologie cristalline,\/} Ann.\ha2Scie.\ha2ENS
              {\bf 12}\ha2(1979),\ha2501--661.}$ 
              
\item[\lit{Ja}] Jannsen, U., {\it Hasse principles for
               higher-dimensional  fields,\/} Annals of Math.\
               {\bf 183} (2016), 1--71. 

\item[\lit{Kh}] Kahn, E.,\ha3{\it La\ha2conjecture\ha2de\ha2Milnor%
\ha2(d'apres\ha2Voevodsky),\/}\ha3S\'em.\ha2Bourbaki,%
\ha2Asterisque\ha2{\bf245}\ha2(1997),\ha{2}379--418.
    
\item[\lit{Ka}] Kato, K., {\it A Hasse principle for two dim global fields}, 
              J.\ha2reine\ha2angew.\ha2Math.\ha3{\bf366}\ha2(1986),\ha2142--180.
              
\item[\lit{K-S}]  Kerz. M.\ and Saito, Sh., {\it Cohomological Hasse 
               principle and motivic cohomology for arithmetic schemes,\/}
               Publ.\hhb2Math.\hhb2IHES {\bf115} (2012), 123--183.

\item[\lit{K-R}] Kim, H.\ha2K.\ and Roush, F.\ha2W., {\it Diophantine 
               undecidability of $\lvC(T_1, T_2)$,\/} J.\ha2Algebra 
               {\bf150} (1992), 35--44.

\item[\lit{Ko1}] Koenigsmann, J., {\it Defining transcendentals 
      in function fields,\/} J. Symbolic Logic {\bf67} (2002), 
      947--956. 
     
\item[\lit{Ko2}] Koenigsmann, J., {\it Defining $\lvZ$ in $\lvQ$,\/} 
              Annals of Math.\ {\bf183} (2016), 73--93.

\item[\lit{Ko3}] Koenigsmann,\ha2J., {\it Decidability in local 
       and global fields,\/} Proc.\ha2ICM\ha22018 Rio\ha2de\ha2Janeiro,
       {\bf Vol.\ha{1}2},\ha263--78.

\item[\lit{M-S}] Merkurjev, A.\ S.\ and Suslin, A.\ A., 
             {\it $K$-cohomology of Severi--Brauer variety and norm
             residue homomorphism,\/} Math.\ USSR Izvestia {\bf 21}
             (1983), 307--340.

\item[\lit{M-Sh}] Miller, R.\ and Shlapentokh, A., {\it On existential 
              definitions of C.E. subsets of rings of functions of 
              characteristic 0,\/} {\tt arXiv:1706.03302} [math.NT]
             
     
\item[\lit{Pf1}] Pfister, A., Quadratic Forms with Applications 
            to Algebraic Geometry and Topology, LMS LNM {\bf217},
            Cambridge University Press 1995; ISBN 0-521-46755-1.

\item[\lit{Pf2}] Pfister, A., {\it On the Milnor conjectures: 
           history, influence, applications,\/} Jahresber.\ha2DMV 
           {\bf102}~(2000), 15--41.

\item[\lit{Pi}] Pierce, D., {\it Function fields and elementary 
           equivalence,\/} Bull.\ha3London Math.\ha3Soc.\ha3{\bf31} 
           (1999), 431--440.


\item[\lit{Po}] Poonen, B., {\it Uniform first-order definitions in 
            finitely gen.\ha2fields,\/} Duke Math.\ J. {\bf 138} (2007), 1--21.

\item[\lit{P-P}] {B.\ Poonen and F.\ Pop, {\it First-order characterization
              of function field invariants over large fields,\/} in: Model Theory
              with Applications to Algebra and Analysis, LMS LNM Series {\bf 350},
              Cambridge Univ.\ Press 2007; pp.~255--271.}

\item[\lit{P1}] Pop, F., {\it Embedding problems over large fields,\/}
              Annals of Math.\ {\bf144} (1996), 1--34.

\item[\lit{P2}] Pop, F., {\it Elementary equivalence versus isomorphism},
             Invent.\ Math.\ {\bf150} (2002), 385--308.

\item[\lit{P3}] 
             Pop, F., {\it Elementary equivalence of finitely generated
             fields,\/} Course Notes Arizona Winter School 2003, see
             {\tt http://swc.math.arizona.edu/oldaws/03Notes.html}

\item[\lit{P4}] 
             Pop, F., {\it Elementary equivalence versus Isomorphisms II,\/}    
             ANT {\bf 11} (2017), 2091-2111. 

\item[\lit{P5}] 
             Pop, F., See {\tt arXiv:1809.00440v1} [math.AG], Sept 3, 2018.

\item[\lit{Ro1}] Robinson, Julia, {\it Definability and decision 
             problems in arithmetic,\/} J.\ Symb.\ha2Logic {\bf14} 
             (1949), 98--114.    

\item[\lit{Ro2}] Robinson, Julia, {\it The undecidability of 
             algebraic fields and rings,\/} Proc.\ha2AMS {\bf10}
             (1959), 950--957. 
             
      
\item[\lit{Ru}]  Rumely, R., {\it Undecidability and Definability 
             for the Theory of Global Fields}, Transactions AMS 
             {\bf262} No.~1, (1980), 195--217.

\item[\lit{Sc}] Scanlon, Th., {\it Infinite finitely generated fields are 
              biinterpretable with ${\bf N}$,\/} JAMS~{\bf21} (2008), 893--908. 
              \vskip0pt\noindent
              {\it Erratum,\/} J.\ Amer.\ Math.\ Soc.~{\bf24} 
              (2011), p.\hhb{2}917.

\item[\lit{Se}] Serre, J.-P., {\it Zeta and L-functions,\/} in: Arithmetical
              Algebraic Geometry, Proc. Conf. Purdue 1963, New York 1965, 
              pp.\ 82--92.

\item[\lit{Sh1}] Shlapentokh, A., {\it First Order Definability and 
        Decidability in Infinite Algebraic Extensions of Rational 
        Numbers,\/} Israel J.\ Math.\ {\bf226} (2018), 579--633.
        
\item[\lit{Sh2}] Shlapentokh, A., {\it On definitions of 
           polynomials over function fields of positive characteristic,\/}  
\item[] See {\tt arXiv:1502.02714v1}        

\item[\lit{Vi}] Vidaux, X., {\it \'Equivalence \'el\'ementaire de 
          corps elliptiques,\/} CRAS S\'erie I {\bf330} (2000), 1-4.
\end{itemize}
}
       

\end{document}